\documentclass[12pt,a4paper]{amsart}

\usepackage{color}

\usepackage[pdfstartview=FitH]{hyperref}
\usepackage[english]{babel}
\usepackage[utf8x]{inputenc}
\usepackage{amsmath, amssymb, amsthm}
\usepackage{graphicx}
\usepackage{enumitem}
\usepackage{soul}
\usepackage[colorinlistoftodos]{todonotes}
\usepackage{xy}
\xyoption{all}

\newif\iftikz
\tikztrue

\iftikz
\usepackage{tikz}
\usetikzlibrary{
calc
}
\fi

\newtheorem{thm}{Theorem}[section]

\newtheorem*{yano}{Yano's Conjecture}
\newtheorem{prop}[thm]{Proposition}
\newtheorem{prop0}{Proposition}
\newtheorem{cor}[thm]{Corollary}
\newtheorem{lema}[thm]{Lemma}

\newtheorem{clm}[thm]{Claim}
\theoremstyle{remark}
\newtheorem{remark}[thm]{Remark}
\theoremstyle{definition}
\newtheorem{dfn}[thm]{Definition}
\newtheorem{ntc}[thm]{Notation}
\newtheorem{ejm}[thm]{Example}

\DeclareMathOperator{\ord}{ord}

\DeclareMathOperator{\spec}{Spec}

\DeclareMathOperator{\Gr}{Gr}

\newcommand\bn{{\mathbb N}}
\newcommand\bz{{\mathbb Z}}

\newcommand\bc{{\mathbb C}}
\newcommand\br{{\mathbb R}}

\newcommand{\cO}{{\mathcal O}}
\newcommand{\cC}{{\mathcal C R_\mu }}
\newcommand{\cA}{{\mathcal E}}
\DeclareMathOperator*{\res}{Res}

\newcommand\enet[1]{\renewcommand\theenumi{#1}
\renewcommand\labelenumi{\theenumi}}

\title[Bernstein polynomial of an irreducible  germ of plane curve]{Bernstein 
polynomial of  2-Puiseux pairs irreducible plane curve singularities
}

\author[E. Artal]{E.~Artal Bartolo${^1}$}
\address{Departamento de Matem\'aticas-IUMA, Universidad de Zaragoza,
c/Pedro Cerbuna 12, 50009 Zaragoza, SPAIN}
\email{artal@unizar.es}
\thanks{${^1}$Partially supported by the grant 
MTM2013-45710-C02-01-P and Grupo Geometr{\'i}a of Gobierno de Arag{\'o}n/Fondo Social Europeo.}
\author[Pi.~Cassou-Nogu\`es]{Pi.~Cassou-Nogu\`es${^2}$}
\address{Institut de Math\'ematiques de Bordeaux, Universit\'e de Bordeaux , 
350, Cours de la Lib\'eration, 33405, Talence Cedex 05, FRANCE}
\email{Pierrette.Cassou-nogues@math.u-bordeaux.fr}
\thanks{${^2}$Partially supported by 
MTM2013-45710-C02-01-P and MTM2013-45710-C02-02-P}
\author[I. Luengo]{I. Luengo${^3}$}
\address{Departamento de \'Algebra, Universidad Complutense,
Plaza de las Ciencias s/n, Ciudad Universitaria, 28040 Madrid, SPAIN}
\curraddr{}
\email{iluengo@ucm.es}

\author[A. Melle]{A.~Melle-Hern\'andez${^3}$}
\address{Departamento de \'Algebra, Universidad Complutense,
Plaza de las Ciencias s/n, Ciudad Universitaria, 28040 Madrid, SPAIN}
\curraddr{}
\email{amelle@ucm.es}

\thanks{${^3}$Partially
supported by the grant 
MTM2013-45710-C02-02-P}



\date{}

\dedicatory{Dedicated  to Prof. H.~B. Laufer on the occasion of his 70th birthday}

\begin{document}

\begin{abstract}
In 1982, Tamaki Yano proposed a conjecture predicting  
the set of $b$-exponents of an irreducible  plane curve singularity germ which is generic 
in its equisingularity class. In \cite{ACLM-Yano2} we proved the conjecture for the case in which the germ 
has two Puiseux pairs 
and its algebraic monodromy has distinct eigenvalues. 
In this article we aim to study the Bernstein polynomial 
for any  function  with the  hypotheses above. 
In particular the set of all common roots of those Bernstein polynomials is given.
We provide also bounds for some analytic invariants of singularities and 
illustrate the computations in suitable examples. 


\end{abstract}

\maketitle

\section*{Introduction}

One  of the main guide lines of Prof. H.B.~Laufer in singularity  theory, 
particularly concerning normal two dimensional analytic singularities $(X,0)$, has been which
analytic invariants of $(X,0)$ depends on the topology, i.e., they are
characterized by their link $L_{(X,0)}$. 
The link has the same information as the decorated resolution graph $\Gamma_{(X;0)}$ see \cite{neu:81}.
For instance Laufer questioned the following  in~\cite{Lau:87}:
What conditions does the existence of a hypersurface representative of $(X,0)$ put on a  
decorated dual graph $\Gamma_{(X,0)}$?
The analytic properties of~$X$ depend on the analytic properties of the ramification locus of a projection.
In this work, we study the behavior of some analytic (non-topological) invariants
for germs of curves.
The main goal of the paper is  to consider germs of irreducible plane curve singularities 
with the same topology and describe exactly the set of common roots of their 
corresponding local Bernstein polynomials which are analytic invariants of their germs.



Let $\cO$ be the ring  of germs of holomorphic functions on $(\bc^n,0)$, 
$\mathcal{D}$ the  ring of germs of holomorphic differential operators of finite order with coefficients in~$\cO$. 
Let $s$ be an indeterminate commuting with the elements of $\mathcal{D}$ and set 
$\mathcal{D}[s]=\mathcal{D}\otimes_{\mathbb{C}} \mathbb{C}[s]$.
 
Given an  holomorphic germ  $f\in \cO$, one considers the ring $R_{f,s}:=\cO\left[\frac{1}{f}, s\right]$
and the free $R_{f,s}$-module $R_{f,s} f^s$ of rank $1$ 
with the  natural  $\mathcal{D}[s]$-module structure. Then,   
there exists a non-zero polynomial $B(s)\in \bc[s]$ and some 
differential operator $P=P(s,x,D)\in\mathcal{D}[s]$, 
holomorphic in $x_1,\dots,x_n$ and polynomial in 
$\frac{\partial }{\partial x_1},\dots,\frac{\partial }{\partial x_n}$, which  satisfies  in 
$R_{f,s} f^s$ the following functional equation
\begin{equation}\label{berstein-rel}
P(s,x,D)\cdot f(x)^{s+1}=B(s)\cdot f(x)^s.
\end{equation}
The monic generator $b_{f,0}(s)$ of the ideal  of such polynomials $B(s)$  is called
the {\it Bernstein polynomial} (or $b$-function or  Berstein-Sato polynomial) of $f$ at $0$.  
The same result holds if we replace $\cO$
by the ring of polynomials in a characteristic zero field ${\mathbb K}$  with the obvious corrections,  see e.g. \cite[10, Theorem 3.3]{Co95}.

This result was first obtained for $f$ polynomial by  Bernstein in~\cite{B72} and in general 
by Bj\"{o}rk \cite{B:81}.
One can  prove  that $b_{f,0}(s)$ is divisible by $s+1$, and we  consider the reduced Bernstein polynomial  
$\tilde{b}_{f,0}(s):=\dfrac{b_{f,0}(s)}{s+1}$.

In the case where $f$  defines an isolated singularity, one can consider 
the Brieskorn lattice  
$H_0^{''}:= \Omega^n /df \wedge d \Omega^{n-2}$ and its saturated  
$ \tilde{H}_0^{''}=\sum_ {k\geq 0}  (\partial _t t)^k      H_0^{''} $.  
Malgrange~\cite{M:75} showed that  
the reduced Bernstein polynomial  
$\tilde{b}_{f,0}(s)$ is the minimal polynomial of the endomorphism $-\partial_t t$ on the vector space 
$F:=\tilde{H}_0^{''}/ \partial_t^{-1}  \tilde{H}_0^{''}$, whose dimension equals  the
 Minor number  $\mu(f,0)$ of $f$ at $0$.
The $b$-{\it exponents}  $\{\beta_1, \ldots,\beta_\mu \}$ are  
the roots of the characteristic polynomial of the    endomorphism $\partial_t t$.
Recall  that $\exp(-2i\pi\partial_t t)$ can be identified with the algebraic monodromy
of the Milnor fiber of $f$ at the singular point.

Kashiwara~\cite{K:76} expressed these ideas with differential operators.
Let $\mathcal{M}:=\mathcal{D}[s]f^s/\mathcal{D}[s]f^{s+1}$,
where $s$ defines an endomorphism of $P(s)f^s$ by multiplication.
This morphism keeps invariant $\tilde{\mathcal{M}}:=(s+1)\mathcal{M}$
and defines a linear endomorphism of $(\Omega^n\otimes_{\mathcal{D}}\tilde{\mathcal{M}})_0$
which is naturally identified with $F$ and under this identification $-\partial_t t$ 
becomes the endomorphism defined by the multiplication by~$s$.

In~\cite{M:75}, Malgrange proved that the set $R_{f,0}$ of roots of the Bernstein polynomial 
is contained in $\mathbb{Q}_{<0}$. Moreover, Kashiwara~\cite{K:76} restricted the set of candidate roots.
The number $-\alpha_{f,0}:=\max R_{f,0}$ 
is the opposite of the log canonical threshold of the singularity.  Saito~\cite{MS:89}
proved that 
\begin{equation}\label{eq:saito}
R_{f,0}\subset[-n+\alpha_{f,0},-\alpha_{f,0}].
\end{equation}


Now let  $f$ be an irreducible germ of  plane curve. In 1982, Tamaki Yano \cite{Y82} made a conjecture concerning the $b$-exponents. 
We state this conjecture in the case we are interested in, the case of two Puiseux pairs.
Let $CS_{n_1,m}^{n_2,q}:=(n_1n_2, mn_2,mn_2+q)$  be the characteristic sequence of $f$, such that

\begin{itemize}
\item $1<n_1<m$, $\gcd(m,n_1)=1$;
\item $q>0$, $n_2>1$, $\gcd(q,n_2)=1$.
\end{itemize}
Recall that this means that $f(x,y)=0$ has as root (say over~$x$) a Puiseux expansion
$$
x=\dots+a_1 y^{\frac{m}{n_1}}+\cdots+a_2 y^{\frac{mn_2+q}{n_1n_2}}+\dots
$$
with exactly~$2$ characteristic monomials.


Let 
\begin{equation}\label{eq:A1} 
B_1:=\left\{\alpha=\frac{m+n_1+k}{mn_1n_2}: 0\leq k< mn_1n_2, 
\text{ and }n_2 m\alpha,n_2 n_1 \alpha\notin\bz\right\};
\end{equation}

\begin{equation}\label{eq:A2} 
B_2:=\left\{\left.\alpha=\frac{\overbrace{(m+n_1)n_2+q+k}^{N_k}}{n_2\underbrace{(mn_1n_2+q)}_{D}}\right\vert0\leq k<n_2D
\text{ and }n_2\alpha,D\alpha\notin\mathbb{Z}\right\}.
\end{equation}


\begin{yano}[\cite{Y82}]  For almost all irreducible plane curve singularity germ $f:(\bc^2,0)\to (\bc,0)$  with characteristic sequence
$(n_1n_2,mn_2,mn_2+q)$, the set $B_f$ of the $b$-exponents  $\{\beta_1,\ldots,\beta_\mu\}$ 
 is $B_1\cup B_2$.
\end{yano}  


In \cite{ACLM-Yano2}  Yano's conjecture was proved 
for the case
\begin{equation}\label{eq:simple_roots}
\gcd(q,n_1)=1\text{ or }\gcd(q,m)=1.
\end{equation}
The above condition is equivalent to require that the algebraic monodromy of the irreducible germ has distinct eigenvalues. 
In this case, the~$\mu$ $b$-exponents are all distinct and they coincide with
the opposite of the roots of the reduced Bernstein polynomial (which turns out to be of degree~$\mu$).

There is another set  which is important too, the set of  the  exponents of the monodromy (or spectral numbers, up to the shift by one, in the terminology of
Varchenko \cite{V82}).  This notion was first introduced by Steenbrink \cite{St}.

Let  $f: (\bc^{n},0) \longrightarrow (\bc ,0)$ be a germ of a holomorphic function with isolated singularity.
In \cite{St} Steenbrink constructed a mixed Hodge structure on $H^{n-1}(F_{f,0},\bc)$. 
Let 
$$H^{n-1}(F_{f,0},\bc)_{\lambda}={\text Ker} (T_s-\lambda:H^{n-1}(F_{f,0},\bc)\longrightarrow H^{n-1}(F_{f,0},\bc));$$
where $T_u,T_s$ are, respectively, the unipotent and semisimple factors of the Jordan decomposition
  of the monodromy $h^{n-1}$. 

The set $\spec(f)$  of spectral numbers are $\mu$ rational numbers 
$$0<\alpha_1\leq \alpha_2\leq \cdots \leq \alpha_{\mu}<n$$
which are defined by the following condition:
\begin{gather*}
\# \{ j:\exp(-2\pi i\alpha_j)=
\lambda, \lfloor\alpha_j\rfloor= n-p-1\}=
\dim _{\mathbb{C}}\Gr_F^p H^{n-1}(F_{f,0},\mathbb{C})_{\lambda},\qquad
\lambda \neq 1\\
\# \{j:\alpha_j=n-p\}=\dim _{\bc}{\text Gr}_F^p H^{n-1}(F_{f,0},\bc)_1.
\end{gather*}
%
The set $\spec(f)$ of spectral numbers  is symmetric, that is $\alpha_i+\alpha_{\mu-(i-1)}=n$. 
 It is known that this set  is constant under $\mu$-constant deformation of $f$, see \cite{V82}.

M. Saito~\cite{MS00} gave a formula for $\spec(f)$  in the case of a germ of an irreducible plane curve singularity
(cf.  also  Theorem~3.1 in \cite{Ne} or section 2.2 in \cite{SSS}).
In the case of characteristic sequence $(n_1n_2,mn_2,mn_2+q)$, the set of spectral numbers 
less than $1$ is the union of the sets 
\begin{gather}\label{eq:a1}
A_1=\left\{
\frac{1}{n_2}\left(\frac {i}{n_1}+\frac{j}{m}\right)+\frac{r}{n_2}\ \Big|\ 0< i<n_1,0<j<m,\frac {i}{n_1}+\frac{j}{m}<1,0\leq r<n_2
\right\},\\
\label{eq:a2}
A_2=\left\{\frac{i}{n_2}+\frac{j}{n_1n_2m+q}\ \Big|\ 0<i<n_2, 0<j<n_1n_2m+q, \frac{i}{n_2}+\frac{j}{n_1n_2m+q}<1\right\}.
\end{gather}

Let us denote by $A_j^\perp:=\{2-\alpha\mid\alpha\in A_j\}$, i.e. the symmetric
set of $A_j$ with respect to~$1$. Then
\begin{equation}\label{eq:union}
\spec(f)=A_1\cup A_2\cup A_1^\perp\cup A_2^\perp.
\end{equation}
There is a closed relationship between spectral numbers and $b$-exponents.
The following result summarizes some of them which can be found for instance in~\cite{HS99}
or~\cite[Remark 3.2 iii)]{MS:93} for~\ref{sp-b-exp1}

\begin{prop0}\label{sp-b-exp:hs}
Let $f$ be a  germ of irreducible plane curve singularity. The spectral numbers $\spec(f) $ and the set $B_f$ of  $b$-exponents  
of $f$ satisfy the following conditions:
\begin{enumerate}
\enet{\rm(\arabic{enumi})}
\item\label{sp-b-exp1}  Let $\alpha_k \in\spec(f)$, there exist a $b$-exponent $\beta_k\in B_f$ such 
that $\alpha_k-\beta_k$ is a non negative integer and $0\leq \alpha_k-\beta_k \leq 1$.
\item $\min B_f=\min\spec(f)=\alpha_1$.
\item  $\dim  \tilde{H}_0^{''}/ H_0^{''}= \sum \alpha_i-\sum \beta_i$
\end{enumerate}
\end{prop0}

From now on, we will study germs 
having a fixed characteristic sequence $CS_{n_1,m}^{n_2,q}$ 
satisfying~\eqref{eq:simple_roots}.
Our goal in this article is to show that one can compute the 
 rational numbers that are roots of the Bernstein polynomial
for any such germ. 
 To do this we follow the same method as the one used  in \cite{Pi88,ACLM-Yano2}. 
To prove that a rational number is a root of the Bernstein polynomial of some function $f$, 
we prove that this number is a pole of some integral with a transcendental residue.  
We also offer algorithmic formul{\ae} for the computation of these residues
and bounds for $\dim  \tilde{H}_0^{''}/ H_0^{''}$.

The two main results in this paper are the following ones.
We split the sets $B_1$ and $B_2$ in terms of two semigroups: $\Gamma$, the one associated with 
$CS_{n_1,m}^{n_2,q}$ (generated by $n_2 n_1, n_2 m, n_1 m n_2+q$)
and~$\Gamma_1$, associated to the \emph{truncation} to the first Puiseux pair (generated by $m,n_1$). Let
\begin{equation}\label{eq:b1}
B_{11}= 
\left\{\beta=\frac{m\beta_1+n_1\beta_2}{mn_1n_2} \in B_1\right|\left. 
\vphantom{\frac{1}{×n}}
\beta_1,\beta_2\in \bz_{\geq 1}\right\}
\end{equation}
(which means that $k$ in \eqref{eq:A1} is in $\Gamma_1$) and
\begin{gather}\label{eq:b2}
B_{21}\!=\!\left\{ \beta=\frac{(m\beta_1+n_1\beta_2)n_2+(mn_1n_2+q)\beta_3+q}{n_2(mn_1n_2+q)}\in B_2\right|
\left. \vphantom{\frac{1}{n}} \beta_1,\beta_2\in\bz_{\geq 1}, \beta_3\in \bz_{\geq 0}
\right\}.
\end{gather}
(which means that $k$ in \eqref{eq:A2} is in $\Gamma_2$).

In Theorem~\ref{cor:common} we prove that
$$
B_{11}\cup B_{21}= \bigcap_{f\in\mathcal{S}_\mu} R_{f,0}
$$
where $\mathcal{S}_\mu$ is the set of all germs $f$ with the  topological type, 
of the characteristic sequence $CS_{n_1,m}^{n_2,q}$ satisfying~\eqref{eq:simple_roots}.

In \S\ref{sec:bounds}, we prove bounds for 
$\dim  \tilde{H}_0^{''}/ H_0^{''}$ for these germs.  
Let
$$
m=\tilde{q}n_1+r_m ,0<r_m<n_1,\quad q=hn_2+r_q,0\leq h,\quad 0<r_q<n_2.
$$ 
Then
$$ (n_2-1)(m-1)(n_1-1) \leq\dim  \tilde{H}_0^{''}/ H_0^{''}\leq \frac{\mu}{2}-n_2(m+n_1)-q+\tilde{q}+h+4;$$
and the second inequality is generically an equality.

We end the article with several families of examples.  
In Theorem~\ref{Yano-family} it is proved that
all polynomials with characteristic sequence $(4,6,6+q)$ have the same Bernstein polynomial
(this is the original Yano's family). 
Next if a polynomial has characteristic sequence 
$(8,10,10+q)$ then we compute its Bernstein polynomial up to six roots 
(note that the Milnor number equals $63+q$) and we have explicit formulae to decide the remaining roots.
And finally,
for the last example, we find differences on the Tjurina and Bernstein stratifications.

\numberwithin{equation}{section}
\section{Two variable integrals and Bernstein polynomial}
\label{sec:res1}

 Let us recall and collect some definitions,  results and consequences from~\cite{ACLM-Yano2}.

\begin{dfn}
We say that a real polynomial $f\in \br [x,y]$ is \emph{positive} if $f(x,y)>0$ for all 
$(x,y)\in[0,1]^2$. 
\end{dfn}
Let $f\in \br [x,y]$ positive.
Let $a_1,a_2, b_1,b_2\in\bz$ be fixed such that $a_1,a_2\geq 0,  b_1,b_2 \geq 1 $. 
We denote the following complex variable integral by
\begin{equation}
\mathcal{Y}(s)=\mathcal{Y}_{f,a_1,b_1,a_2,b_2}(s):=
\int_0^1  \int_0^1  f(x,y)^s x^{a_1s+b_1} y^{a_2s+b_2} \frac{dx}{x} \frac{dy}{y}.
\end{equation}

\begin{prop}[Proposition~1.4~\cite{ACLM-Yano2}]\label{Essou1a} The function $\mathcal{Y}(s)$ satisfies the following properties:
\begin{enumerate}
\enet{\rm(\arabic{enumi})} 
 \item  It is absolutely convergent for $\Re(s)>\alpha_0$, where $\alpha_0=
\sup\left(-\frac{b_1}{a_1},-\frac{b_2}{a_2}\right)$
\item  It has a meromorphic continuation on $\bc$ with poles of order at most $2$ contained in 
$S=\left\{ -\frac{b_1+\nu_1}{a_1}, \, \, \nu_1 \in \mathbb{Z}_{\geq 0} \,\right\} \cup 
\left\{ -\frac{b_2+\nu_2}{a_2}, \, \, \nu_2\in \mathbb{Z}_{\geq 0} \,\right\}$
\end{enumerate}

\end{prop}

\begin{ntc}\label{mer-cont}
Let $f: [0,1]\to\br$ be a continous function. 
We will denote by $G_{f}(s)$ the meromorphic continuation of
$$
\int_0^1 f(t)t^s\frac{dt}{t}.
$$
\end{ntc}

\begin{prop}[Proposition~1.6~\cite{ACLM-Yano2}]\label{continuation} 
With the hypotheses of Proposition{\rm~\ref{Essou1a}},
let $\nu_1\in \mathbb{Z}_{\geq 0}$ be fixed and such that $\alpha=-\frac{b_1+\nu_1}{a_1}\neq -\frac{b_2+\nu_2}{a_2}$ 
for all $\nu_2\in \mathbb{Z}_{\geq 0}$, then the pole  of $\mathcal{Y}(s)$  at $\alpha$ is simple and 

\begin{equation}
\res_{s=\alpha} {\mathcal{Y}(s)}=\frac{1}{\nu_1! a_1} 
G_{ h_{\nu_1,\alpha,x}}(a_2\alpha+b_2),\quad
h_{\nu_1,\alpha,x}(y):=\frac{\partial^{\nu_1} f^\alpha}{\partial x^{\nu_1}}(0,y).
\end{equation}
\end{prop}

Note that, under the hypotheses of the above Proposition, 
$G_{ h_{\nu_1,\alpha, x}}(a_2 s+b_2)$ admits an integral  expression which
is absolutely convergent and holomorphic
for  $\Re(s)> -N_2-1$,  with $N_2$ such that $\alpha>-  \frac{b_2+N_2+1}{a_2}$,
see the proof in~\cite{ACLM-Yano2} of the above Proposition~\ref{Essou1a} .


We collect next  a result which relates these integrals with the  beta function $\boldsymbol{B}(s_1,s_2)$.

\begin{lema}[Lemma~1.8~\cite{ACLM-Yano2}]\label{beta}
Let $p\in \mathbb{N}$ and $c\in \mathbb{R}_{>0}$. 
Given $s_1, s_2\in \bc$ 
  such that  $-\alpha=s_1+s_2>0$
then 
\begin{equation}
G_{ \left( y^{p}+c\right)^\alpha}(p s_1)+
G_{ \left(1 +c x^{p}\right)^\alpha}(ps_2)=
\frac{c^{-s_2}}{p}\boldsymbol{B}\left( s_1, s_2\right)
\end{equation}
where $\boldsymbol{B}$ is the \emph{beta function}.
\end{lema}

Our goal in this article is to show that one can compute the 
 rational numbers that are roots of the Bernstein polynomial
for any function having characteristic sequence $CS_{n_1,m}^{n_2,q}$ satisfying~\eqref{eq:simple_roots}.
To do this we follow the same methods and ideas as the one used by Pi.~Cassou-Nogu\`es in \cite{Pi861,Pi862,Pi872,Pi88,ACLM-Yano2}. 
To prove that a rational number is a root of the Bernstein polynomial of some function $f$, 
we prove that such a number is pole of some integral whose residue is a transcendental number.  

To use the method one needs to start with a real polynomial $f\in \br [x,y]$ 
whose complex analytic germ at the origin has $CS_{n_1,m}^{n_2,q}$ as characteristic sequence. 

\begin{dfn} A polynomial $f\in\br[x,y]$ 
is said to be of \emph{type $(n_1n_2, mn_2,mn_2+q )^+$ }
if it satisfies:
\begin{equation}\label{eq:f_gen1}
f(x,y)=(x^{n_1}+ y^{m}+h_1(x,y))^{n_2}+ x^a y^b+h_2(x,y)
\end{equation}
where
\begin{enumerate}
\enet{(G$^+$\arabic{enumi})}
\item\label{G1+} $ h_1(x,y)=\sum_{(i,j)\in\mathcal{P}_{n_1,m}} a_{i j} x^i y^j\in\mathbb{R}[x,y]$,
where
$$
\mathcal{P}_{n_1,m}:=\{(i,j)\in\mathbb{Z}_{\geq 0}^2\mid m i+n_1 j> m n_1\};
$$
\item\label{G2+} $a,b\geq 0$ such that $a m+b n_1=mn_1n_2+q$;
\item\label{G3+} the polynomial $h_2\in\mathbb{R}[x,y]$, whose support is disjoint from the other terms of $f$,
satisfies that the characteristic sequence of $f$ is $CS_{n_1,m}^{n_2,q}$.
\end{enumerate}
\end{dfn}

\begin{prop}\label{123to4}
 Let $f\in\br[x,y]$ be a real polynomial as in \eqref{eq:f_gen1}    
 satisfying \eqref{G1+}, \eqref{G2+} and ~\eqref{G3+}.
Then there exists a domain  $D=[0,\eta]^2$, with $\eta\leq 1$, such that $f>0$ in $[0,\eta]^2\setminus\{(0,0)\}$.
\end{prop}
\begin{proof}
Note that the real zero locus of $x^{n_1}+y^m$ intersects $[0,1]^2$ only at $(0,0)$. Since
the real zero locus of $f=0$ is a deformation of the previous one, then 
there is $\eta>0$ for which the statement follows. 
\end{proof}

For $\beta_1,\beta_2\in\bz_{\geq 1}$, and $f$ of type $(n_1n_2, mn_2,mn_2+q )^+$  one defines:
\begin{equation}
I_+(f,\beta_1,\beta_2)(s):=\int_0^1  \int_0^1  f(x,y)^s\, x^{\beta_1} 
y^{ \beta_2}\, \frac{dx}{x}\, \frac{dy}{y}.
\end{equation}

\begin{prop}[\cite{ACLM-Yano2}]\label{prop:newton1}
Let $f$ be of type $(n_1n_2, mn_2,mn_2+q )^+$
and $\beta_1,\beta_2\in\bz_{\geq 1}$. Then the integral $I_+(f,\beta_1,\beta_2)(s)$ is  absolutely 
convergent for $\Re{(s)}>-\frac{\beta_1m+\beta_2n_1}{mn_1n_2}$
and may have simple poles only for $s=-\frac{\beta_1m+\beta_2n_1+\nu}{mn_1n_2}, \nu\in \bz_{\geq 0}$.
\end{prop}

Next we show   the algorithmic description of \cite[Section 3]{ACLM-Yano2} to compute the residue of the corresponding  
family of poles. 
Let us see show to compute the residue at the eventual pole
$\alpha=-\frac{\beta_1m+\beta_2n_1+\nu}{mn_1n_2}$ of the integral $I_+(f,\beta_1,\beta_2)(s)$. Let
$$
\tilde{f}(x,y)=f(x^m,y^{n_1})
$$
and let $f_1$ and $f_2$ be defined by
$$
\tilde{f}(x,xy)=x^{n_1n_2m}f_1(x,y),\qquad
\tilde{f}(xy,y)=y^{n_1n_2m}f_2(x,y).
$$

Thus the residue of $\alpha=-\frac{\beta_1m+\beta_2n_1+\nu}{mn_1n_2}$ of the integral $I_+(f,\beta_1,\beta_2)(s)$ equals
\begin{equation} \label{residue-B1}
 \res_{s=\alpha} I_+(f,\beta_1,\beta_2)(s)=\frac{1}{\nu !m n_1 n_2}
(G_{ h^1_{\nu,\alpha,x}}(n_1\beta_2)+G_{ h^2_{\nu,\alpha,y}}(m\beta_1));
\end{equation}
 
where 
$$
h^1_{\nu,\alpha,x}(y)=\frac{\partial^{\nu}f_1^{\alpha}}{\partial x^{\nu}}(0,y),\qquad \text{and}\qquad 
h^2_{\nu,\alpha,y}(x)=\frac{\partial^{\nu}f_2^{\alpha}}{\partial y^{\nu}}(x,0).
$$

We define now a simplified version of polynomials of type $(n_1 n_2, mn_2,mn_2+q )^-$
defined in~\cite{ACLM-Yano2}.


\begin{dfn} 
A polynomial $f\in \br[x,y]$
is  said to be of type $(n_1 n_2, mn_2,mn_2+q )^-_s$ 
if it satisfies:
\begin{equation}\label{eq:f_gen2}
f(x,y)=g(x,y)^{n_2}+ x^a y^b+h_2(x,y)
\end{equation}
where 
$g(x,y):=x^{n_1}-y^{m}$
\begin{enumerate}
\enet{(G$^{-}$\arabic{enumi})}
\item\label{G2-} $a,b\geq 0$  are as in~\ref{G2+}.
\item\label{G4-} The polynomial $h_2\in\mathbb{R}[x,y]$, whose  support  is disjoint from the first terms, 
satisfies  that the characteristic sequence of $f$ is $CS_{n_1,m}^{n_2,q}$.
\item\label{G5-} There is an $\epsilon>0$ such that for 
$\mathcal{D}:= \{(x,y)\in\br^2\mid 0\leq x\leq\epsilon, 0\leq y\leq x^{\frac{n_1}{m}}\}$,
we have that $f> 0$ on $\mathcal{D} \setminus \{(0,0)\}$.
\end{enumerate}
\end{dfn}

\begin{prop}\label{prop:f-}
For each $f$ as in \eqref{eq:f_gen2} satisfying the conditions {\rm~\ref{G2-}}, and{\rm~\ref{G4-}}  then there is 
$\epsilon>0$ and a domain $\mathcal{D}:= \{(x,y)\in\br^2\mid 0\leq x\leq\epsilon, 0\leq y\leq x^{\frac{n_1}{m}}\}$
for which $f$ satisfies 
 {\rm~\ref{G5-}} in $D$, that is $f$  is of type $(n_1 n_2, mn_2,mn_2+q )^-_s$. 
\end{prop}

\begin{proof}
It is enough to take a suitable truncation of a Puiseux expansion of $f$
(which has no term between the two characteristic terms).
\end{proof}

For $\beta_1,\beta_2\in\bz_{\geq 1}$,  $ \beta_3\in \mathbb{Z}_{\geq 0}$ and $f$ of type $(n_1n_2, mn_2,mn_2+q )^-_s$
we set:
\begin{equation}
\mathcal{I}_{-}(f,\beta_1,\beta_2,\beta_3)(s):=\int\!\!\!\int_{\mathcal{D}}  f(x,y)^s\, x^{\beta_1} 
y^{ \beta_2}g(x,y)^{\beta_3}\, \frac{dx}{x}\, \frac{dy}{y}.
\end{equation}

\begin{prop}[\cite{ACLM-Yano2}]\label{prop:newton2}
Let $f\in\br[x,y]$ be a polynomial of type $(n_1 n_2, mn_2,mn_2+q )^-_s$,
$\beta_1,\beta_2\in\mathbb{Z}_{\geq 1}$ and 
$\beta_3\in\mathbb{Z}_{\geq 0}$.
Then the integral $\mathcal{I}_-(f,\beta_1,\beta_2,\beta_3)(s)$ is convergent for 
$\Re{(s)}>-\frac{\beta_1m+\beta_2n_1+\beta_3mn_1}{mn_1n_2}$ and
its set of poles is contained in the set 
$$
P_{1}\cup\bigcup_{i\in\mathbb{Z}_{\geq 1} } P_{2,i}
$$
where
$$
P_{1}:=\left\{\left.-\frac{m\beta_1+n_1\beta_2+m n_1\beta_ 3+\nu}{mn_1n_2}\right|\nu\in\mathbb{Z}_{\geq0}\right\}
$$
and
$$
P_{2,i}:=\left\{\left.-\frac{n_2(m \beta_1+n_1 \beta_2+m n_1\beta_3+)+q(\beta_3+i)+\nu}{n_2(mn_1n_2+q)}\right|\nu\in\mathbb{Z}_{\geq0}\right\}
$$
The poles have at most order two. The poles may have order two at the values contained in  
$P_{1} $ and $P_{2,i}$ for some $i$.
\end{prop}

We shall give the residues at the eventual simple poles in $P_{2,i}$.
Let $\tilde{f},\tilde{\tilde{f}},\hat{f}$ be defined by 
$$
f(x^m,y^{n_1})=\tilde{f}(x,y),\quad 
\tilde{f}(x,xy)=x^{mn_1n_2}\tilde{\tilde{f}}(x,y),\quad
\hat{f}(x,y)=\tilde{\tilde{f}}(x,1-y).
$$
Let $\hat{\hat{f}},f_1,f_2$ be defined by 
$$
\hat{f}(x^{n_2},y^q)=\hat{\hat{f}}(x,y),\quad 
\hat{\hat{f}}(x,xy)=x^{n_2q}f_1(x,y),\quad 
\hat{\hat{f}}(xy,y)=y^{n_2q}f_2(x,y).
$$
Let us denote
$$
g(x^m,y^{n_1})=\tilde{g}_Y(x,y)=x^{m n_1}-y^{m n_1},\qquad 
\tilde{\tilde{g}}(y)=\frac{\tilde{g}(x,xy)}{x^{mn_1}}=1-y^{m n_1}.
$$
In particular, 
$$
\tilde{\tilde{g}}(1-y)=y Q(y),\quad Q(0)=n_1.
$$
Let us define
$$
\tilde{Q}(y)=Q(y)^{\beta_3}(1-y)^{n_1\beta_2-1},\quad
\tilde{Q}(y)=\sum b_{i}y^{i-1}.
$$
Thus  the integral $\mathcal{I}_-(f,\beta_1,\beta_2,\beta_3)(s)$
has residue for 
$$
s=\alpha=-\frac{n_2(m \beta_1+n_1 \beta_2+m n_1\beta_3)+q(\beta_3+i)+\nu}{n_2(mn_1n_2+q)}
$$ 
equals
\begin{equation}\label{residue-B2}
\begin{gathered}
 \res_{s=\alpha} \mathcal{I}_-(f,\beta_1,\beta_2,\beta_3)(s)=
\frac{1}{n_2q}\sum_{i,\nu} \frac{1}{\nu!} b_{i} 
(G_{h^1_{\nu,\alpha,x}}(q(\beta_3+i))+\\
G_{h^2_{\nu,\alpha,y}}(n_2(mn_1n_2\alpha+m\beta_1+n_1\beta_2+mn_1\beta_3))
\end{gathered}
\end{equation}
where 
$$
h^1_{\nu,\alpha,x}(y)=\frac{\partial^{\nu}f_1^{\alpha}}{\partial x^{\nu}}(0,y),\qquad \text{and} \qquad
h^2_{\nu,\alpha,y}(x)=\frac{\partial^{\nu}f_2^{\alpha}}{\partial y^{\nu}}(x,0);
$$
recall also that $G_f(s)$ is the meromorphic continuation of 
$
\int_0^1 f(t)t^s\frac{dt}{t}.
$

\begin{remark}
We may assume $\epsilon=1$ after a suitable change of variables.
\end{remark}

 Let us summarize the links between these integrals and the Bernstein polynomial.
We are using ideas in \cite{Pi862,Pi861,Pi872,ACLM-Yano2}. 
Let us fix notations that may cover both cases.
We fix $f$ with the following properties:
\begin{enumerate}
\enet{(B\arabic{enumi})}
\item The characteristic sequence of $f\in\br[x,y]$ is $CS_{n_1,m}^{n_2,q}$.
\item The polynomial $Y(x^{\frac{1}{m}})\in\br[x^{\frac{1}{m}}]$  is either~$1$ (for the $+$-case)
or $x^{\frac{n_1}{m}}$ for the $-_s$-case
\item $\mathcal{D}:=\{(x,y)\in\br^2\mid 0\leq x\leq 1,0\leq y\leq Y(x^{\frac{1}{m}})\}$, $g(x,y)=x^{n_1}\pm y^m$.
\item $f(x,y)>0$ $\forall (x,y)\in\mathcal{D}\setminus\{(0,0)\}$.
\end{enumerate}

Let $\beta_1,\beta_2\in\mathbb{Z}_{\geq 1}$ and $\beta_3\in\mathbb{Z}_{\geq 0}$ (equals~$0$ for the $+$-case). 
Let us consider
the integral
\begin{equation}
\mathcal{I}_\pm(f,\beta_1,\beta_2,\beta_3)(s):=
\int\!\!\!\int_{\mathcal{D}}  f(x,y)^s\, x^{\beta_1} y^{\beta_2}\, g(x,y)^{\beta_3}\frac{dx}{x}\, \frac{dy}{y}.
\end{equation}
\begin{thm}[{\cite[Theorem 5.3]{ACLM-Yano2}}]\label{pole-integral-root} Let $f(x,y)\in \mathbb{K}[x,y]$ be a polynomial defining an  irreducible germ of complex plane curve at the origin 
which has two Puiseux pairs and its algebraic monodromy has distinct eigenvalues and such that 
$\mathbb{K}$ is an algebraic extension of~$\mathbb{Q}$. 
Let $\alpha$ be a pole of  $\mathcal{I}_{\pm}(f,\beta_1,\beta_2,\beta_3)(s)$  with transcendental residue,  
and such that $\alpha+1$ is not a pole of $\mathcal{I}_{\pm}(f,\beta'_1,\beta'_2,\beta'_3)(s)$ 
for any $(\beta'_1,\beta'_2,\beta'_3)$. Then $\alpha$ is root of
the Bernstein-Sato polynomial  $b_{f}(s)$ of $f$.
\end{thm}

\section{\texorpdfstring{Determination of the set of common roots of the $\mu$-constant stratum}{Determination 
of the set of common roots of the mu-constant stratum}}\label{common-roots}
Let $f$ be an irreducible germ of plane curve whose characteristic sequence is~$CS_{n_1,m}^{n_2,q}$
satisfying~\eqref{eq:simple_roots}.
The Bernstein-Sato polynomial of a germ~$f$ with this 
characteristic sequence, depends on~$f$, but there is a \emph{generic} Bernstein polynomial~$b_{\mu, \text{gen}}(s)$: 
for every $\mu$-constant deformation of such an~$f$, there is a Zariski dense open set\ $\mathcal{U}$ on which the Bernstein-Sato polynomial
of any germ in~$\mathcal{U}$ equals $b_{\mu,gen}(s)$. 

\begin{prop}[{\cite[Corollary 21]{V80}}]\label{semi} Let $f_t(x)$  be a $\mu$-constant analytic deformation of  an  
isolated hypersurface singularity $f_0(x)$. 
If all  eigenvalues of  the monodromy are pairwise different, then all roots
of the reduced Bernstein-Sato polynomial $\tilde{b}_{f_t}(s)$ depend lower semi-continously upon the parameter $t$.
\end{prop}

\begin{prop} [{\cite[Corollary 5.1]{HS99}},{\cite{GH07}}]   \label{mu-cons}
Let $f(x)$  be a germ of an 
isolated hypersurface singularity. 
Then for each spectral number $\alpha \in\spec(f)$ such that  $\alpha< \alpha_1 +1,$ 
then 
$-\alpha$ is root of the Bernstein polynomial ${b}_{f}(s)$.

Consequently, for a $\mu$-constant analytic deformation $f_t(x)$  of  an  
isolated hypersurface singularity germ $f_0(x)$, 
for every $\alpha$ in
$$
\cA:=\{\alpha: \alpha \in\spec(f)\text{ and }\alpha< \alpha_1 +1 \}
$$
then 
$-\alpha$ is root of every    Bernstein polynomial ${b}_{f_t}(s)$ for every $t$.
\end{prop}

\begin{remark}
Note that we follow Saito's convention for the exponents and the spectral numbers, which differs
by $1$ from the convention in~\cite{HS99}.
\end{remark}


The following Corollary is a consequence of Proposition{\rm~\ref{semi}}.


\begin{cor}\label{varchenko-gen}
Let $f_0(x,y)$ be an irreducible germ of plane curve whose 
monodromy has distinct eigenvalues. Let $-\alpha$ be a root 
of the local Bernstein-Sato polynomial $b_{f_0}(s)$. Then, either $-\alpha$ or $-(\alpha+1)$
is a root of $b_{\mu,gen}(s)$. 
\end{cor}

Let $S_\mu$ be the (non-singular) $\mu$-constant stratum of $f$ at $0$. 
Let $R_{f}$ be the set  
of the roots of $b_{f}(-s)$. For every $g\in S_\mu$ and since $g$ has isolated singularities 
then
$$
R_g\subset \spec(g) \cup \{\alpha-1\mid\alpha\in\spec(g)\}\text{, see \cite{GH07}}.
$$ 
Since the spectral numbers  are constant 
in a $\mu$-constant deformation then, in the image of the map 
$ S_\mu \to \bc [s]: g \mapsto b_{g}(s)$  there are finitely many polynomials. 

The aim of this section is to describe 
the set of common roots of the Bernstein polynomials of the $\mu$-constant stratum, that is, the set  
$$
\cC:=\bigcap_{f \in S_\mu} R_{f}.
$$
By Proposition~(\ref{mu-cons}) the set $\cA \subset \cC$.

In \cite{ACLM-Yano2},  we proved that the set of roots of the Bernstein polynomial~$b_{\mu, \text{gen}}(-s)$ is 
$B_1\cup B_2$.
We split these sets $B_1$ and $B_2$ using~\eqref{eq:b1} and~\eqref{eq:b2} and we set
%
%
$B_{12}=B_1\setminus B_{11}, B_{22}=B_2\setminus B_{21}$. 



The aim of this part is to prove 

\begin{thm}\label{cor:common}
Let $\cC$ be the set of common roots of the Bernstein polynomials of every  
irreducible germ of plane curve whose characteristic sequence is~$CS_{n_1,m}^{n_2,q}$ satisfying~\eqref{eq:simple_roots}.
Then $\cC=B_{11} \cup B_{21}.$
\end{thm}

We divide the proof in three parts.

\begin{prop}
$B_{11} \cup B_{21} \subset \cC.$
\end{prop}

\begin{proof}
Let $\alpha \in B_{11} \cup B_{21} \subset B_1 \cup B_2$ and let $f $ be a fixed germ in $S_\mu$. 
As $-\alpha$  is a root of the generic Bernstein polynomial,  if $-\alpha $ is not a root of the Bernstein polynomial of $f$, 
then $-(\alpha+1)$ is by Corollary~\ref{varchenko-gen}.
Then $-\alpha -1>-2$ and $\alpha <1$. In particular, if $\alpha>1$ then
$\alpha$ is a root for any germ. We need only to study $B'=(B_{11} \cup B_{21})\cap \{\alpha\in \mathbb{Q}\vert \alpha<1\}$.
Hence, we need only to prove that $B'\subset \cC$.

To do this, we use a result of B. Lichtin 
(\cite[Section 3, Corollary 2]{Li89}) and Loeser (\cite[Remarque III.3.5]{Lo88}).  
Let us state it.

\begin{figure}[ht]
\begin{center}
\iftikz
\begin{tikzpicture}[vertice/.style={draw,circle,fill,minimum size=0.2cm,inner sep=0}]
\coordinate (M1) at (0,0);
\coordinate (M2) at (3.6,0);
\node[vertice] at (M1) {};
\draw ($(M1)-(1.5,0)$)--(M1);
\node[vertice] at ($(M1)+(-1,0)$) {};
\node at ($(M1)+(-1.75,-.01)$) {$\dots$};
\draw ($(M1)-(0,1.5)$)--(M1);
\node[vertice] at ($(M1)+(0,-1)$) {};
\node at ($(M1)+(0.01,-1.65)$) {$\vdots$};
\draw ($(M1)+(1.5,0)$)--(M1);
\node[vertice] at ($(M1)+(1,0)$) {};
\node at ($(M1)+(1.85,-.01)$) {$\dots$};
\node[vertice] at (M2) {};
\draw ($(M2)-(1.5,0)$)--(M2);
\node[vertice] at ($(M2)+(-1,0)$) {};
\draw ($(M2)-(0,1.5)$)--(M2);
\node[vertice] at ($(M2)+(0,-1)$) {};
\node at ($(M2)+(0.01,-1.65)$) {$\vdots$};
\draw[->] (M2)--($(M2)+.75*(1,1)$);

\node[above] at (M1) {$D_1$};
\node[above] at ($(M1)-(1,0)$) {$D_{1,1}$};
\node[right] at ($(M1)-(0,1)$) {$D_{1,2}$};
\node[above] at ($(M1)+(1,0)$) {$D_{1,3}$};
\node[below right] at (M2) {$D_2$};
\node[above] at ($(M2)-(1,0)$) {$D_{2,1}$};
\node[left] at ($(M2)-(0,1)$) {$D_{2,2}$};
\node[above] at ($(M2)+(1,0)$) {$D_{2,3}$};
\end{tikzpicture}
\else
\includegraphics{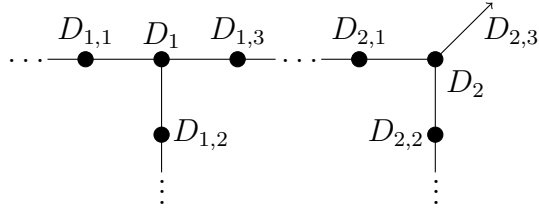}
\fi
\caption{Resolution graph}
\label{fig:grafo}
\end{center}
\end{figure}

Let us consider an embedded resolution of~$f$, see Figure~\ref{fig:grafo}, together with
a  $2$-differential form~$\omega$. For a divisor $D_j$,
let us denote $N_j:=\ord_{D_j}(f)$ and $\nu_j(\omega):=\ord_{D_j}(\omega)+1$.

In the resolution of $f$ we have two branching divisors that we denote by $D_1,D_2$. 
We denote by $D_{j,1},D_{j,2},D_{j,3}$ the divisors adjacent to $D_j,j=1,2$. For $j=1,2,k=1,2,3$ we set
$$
\epsilon_{j,k}(\omega):=
\frac{1}{N_j}\det
\begin{pmatrix}
N_j&N_{j,k}\\
\nu_j(\omega)&\nu_{j,k}(\omega)
\end{pmatrix}
\overset{\bmod\bz}{\equiv}
-\frac{\nu_j(\omega)N_{j,k}}{N_j}.
$$
If the following conditions hold:
\begin{itemize}
\item $\epsilon _{j,k}(\omega)$ is not an integer for $k=1,2,3$,
\item $\frac{\nu_j(\omega)}{N_j}<1$;
\end{itemize}
then  $-\frac{\nu_j(\omega)}{N_j}$
is a root of the Bernstein polynomial of $f$. 
We have the relation $\sum _k\epsilon _{j,k}(\omega)=-2$ for $j=1,2$.

It is easily checked that
$$
N_1=m n_1 n_2,\quad N_2=n_2(m n_1 n_2+q),
$$
and 
$$
N_{1,1}=n_1 n_2 \alpha_{1,1},\quad N_{1,2}=m n_2 \alpha_{1,2},\quad N_{1,3}=(m n_1 \alpha'_{1,3}+1)n_2,
$$
where
$$
\alpha_{1,1} n_1+1=\alpha'_{1,1}m,\quad
\alpha_{1,2} m+1=\alpha'_{1,2}n_1,\quad
\alpha_{1,3}+n_2=\alpha'_{1,3} q,
$$
all positive integers. We also have
$$
N_{2,1}=n_2(m n_1 \alpha'_{2,1}+\alpha_{2,1}),\quad 
N_{2,2}=\alpha_{2,2}(m n_1 n_2+q),\quad N_{2,3}=1, 
$$
where
$$
\alpha_{2,1} n_2+1=\alpha'_{2,1}q,\quad
\alpha_{2,2} q+1=\alpha'_{2,2}n_2,
$$
again all positive integers.

Let us check the conditions for $\alpha=\frac{m\beta_1+n_1\beta_2}{mn_1n_2}\in B_{11}\cap B'$
with the form $\omega_1=x^{\beta_1-1}y^{\beta_2-1}dxdy$. 
Since $\nu_1(\omega_1)=m\beta_1+n_1\beta_2$, we have that $\alpha=\frac{\nu_1(\omega_1)}{N_1}$.
Only the non-integer condition must be checked.
We have:
\begin{gather*}
-\epsilon_{1,1}(\omega_1)\overset{\bmod\bz}{\equiv}\frac{(m\beta_1+n_1\beta_2)\alpha_{1,1}n_1 n_2}{m n_1 n_2}
\overset{\bmod\bz}{\equiv}\frac{n_1\beta_2\alpha_{1,1}}{m}
\overset{\bmod\bz}{\equiv}-\frac{\beta_2}{m}\notin\bz\\
-\epsilon_{1,2}(\omega_1)\overset{\bmod\bz}{\equiv}\frac{(m\beta_1+n_1\beta_2)\alpha_{2,1}m n_2}{m n_1 n_2}
\overset{\bmod\bz}{\equiv}\frac{m\beta_1\alpha_{1,2}}{n_1}
\overset{\bmod\bz}{\equiv}-\frac{\beta_1}{n_1}\notin\bz\\
-\epsilon_{1,3}(\omega_1)\overset{\bmod\bz}{\equiv}\frac{(m\beta_1+n_1\beta_2)(m n_1 \alpha'_{1,3}+1) n_2}{m n_1 n_2}
\overset{\bmod\bz}{\equiv}\frac{m\beta_1+n_1\beta_2}{m n_1}\notin\bz
\end{gather*}

Finally we check the conditions for 
$\alpha=\frac{(m\beta_1+n_1\beta_2)n_2+q+(mn_1n_2+q)\beta_3}{n_2(mn_1n_2+q)}\in B_{21}\cap B'$
with the form $\omega_2=x^{\beta_1-1}y^{\beta_2-1}g_Y(x,y)^{\beta_3}dxdy$. 
Since $\nu_2(\omega_2)=(m\beta_1+n_1\beta_2)n_2+q+(mn_1n_2+q)\beta_3$, we have that $\alpha=\frac{\nu_2(\omega_2)}{N_2}$.
Let us check the non-integer condition.
We have:
\begin{gather*}
-\epsilon_{2,1}(\omega_2)\overset{\bmod\bz}{\equiv}
\frac{((m\beta_1+n_1\beta_2)n_2+q+(mn_1n_2+q)\beta_3)(\alpha'_{2,1}m n_1+\alpha_{2,1}) n_2}{n_2(mn_1n_2+q)}
\overset{\bmod\bz}{\equiv}\\
-\frac {m\beta_1+n_1\beta_2-mn_1}{mn_1n_2+q}\\
\epsilon_{2,2}(\omega_2)\!\!\!\overset{\bmod\bz}{\equiv}\!\!\!\!-
\frac{((m\beta_1+n_1\beta_2)n_2+q+(mn_1n_2+q)\beta_3)\alpha_{2,2}(m n_1 n_2+q)}{n_2(mn_1n_2+q)}\!\!\!
\overset{\bmod\bz}{\equiv}\!\!\frac{\beta_3+1}{n_2}\\
\epsilon_{2,3}(\omega_2)\!\!\!\overset{\bmod\bz}{\equiv}\!\!\!\!-
\frac{((m\beta_1+n_1\beta_2)n_2+q+(mn_1n_2+q)\beta_3)}{n_2(mn_1n_2+q)}.
\end{gather*}
None of the above numbers is an integer.
\end{proof}

\begin{prop}\label{prop:B12}
For all  $\beta \in B_{12}$, there exits $f_{\beta}$ with characteristic sequence $(n_1n_2,mn_2,mn_2+q)$ 
such that $-\beta$ is  not a root of the Bernstein polynomial of~$f_{\beta}$.
\end{prop}

\begin{proof}
Let $\beta \in B_{12}$, i.e. there exists $k\geq 1$ 
such that $\beta=\frac{n_1+m+k}{n_1n_2 m}$  and $k$ is not
in the semigroup $\Gamma_1$ generated by $m,n_1$. 
Then $\beta+1=\frac{n_1+m+n_1 m m_2 +k}{n_1n_2 m}$.
Since the conductor of $\Gamma_1$ is $m n_1-m-n_1$, then
$n_1+m+n_1 m m_2 +k\in\Gamma_1$ and 
there exist $\beta_1$ and $\beta_2$ such that $1+\beta =\frac{m\beta_1+n_1\beta_2}{mn_1n_2}$.
In particular, if $f$ is of type $(n_1n_2,mn_2,mn_2+q)^+$, then
$$ 
 \res_{s=-\beta-1} I(f,\beta_{1},\beta_{2})(s)\neq 0
$$
is transcendental, see~\cite[Proposition~3.3]{ACLM-Yano2}.

\begin{clm}\label{claim1}
There exists $f$ of type $(n_1n_2,mn_2,mn_2+q)^+$  such that   
$$ 
 \res_{s=-\beta} I_+(f,\beta_1,\beta_2 )(s)=0, \qquad  
\qquad \forall(\beta_1,\beta_2)\in \bz_{\geq1}. 
$$
\end{clm}

Assume that Claim~\ref{claim1} has been proved (see the Appendix~\ref{sec:appendix}).
For such an~$f$,  $-\beta-1$ is a root of the Bernstein polynomial of $f$.
Since the hypotheses of Theorem~\ref{pole-integral-root} hold (see also \cite[Theorem~5.3]{ACLM-Yano2})
$-\beta$ is not such a root since the monodromy has distinct eigenvalues. Then there exists $f$
such that  $-\beta $ is not root of the  Bernstein  polynomial $b_{f,0}(s)$.
\end{proof}

\begin{prop}\label{prop:B22}
For all $\beta \in B_{22}$, there exists $f_{\beta}$ 
with characteristic sequence $(n_1n_2,mn_2,mn_2+q)$ such that $-\beta$ 
is not a root of the Bernstein polynomial of $f_{\beta}$.
\end{prop}
\begin{proof}
Let $\beta \in B_{22}$, i.e  we have $\beta=\frac{n_2(m+n_1)+q+k}{n_2(mn_1n_2+q)}$ and
$k\notin\Gamma$, where $\Gamma$ is the semigroup generated by $m n_2$, $n_1 n_2$
and $mn_1n_2+q$; its conductor is
$$
n_2 (m n_{1} n_{2}+q) - (m+ n_{1}) n_{2}  - q + 1
$$ 
In particular,
$n_2 (m n_{1} n_{2}+q)+k\in\Gamma$ and there exist
$\beta_1,\beta_2,\beta_3$ such that 
$$
1+\beta =\frac{m\beta_1+n_1\beta_2+(mn_1n_2+q)\beta_3}{n_2(mn_1n_2+q)}.
$$
As a consequence, for any $f$ of type $(n_1n_2,mn_2,mn_2+q)^-$, we have that
$$ 
\res_{s=-\beta-1} I_-(f,\beta_{1},\beta_{2},\beta_3)(s)\neq 0
$$
is transcendental, see~\cite[Proposition~4.3]{ACLM-Yano2}.

\begin{clm}\label{claim2}
There exists $f$ of type $(n_1n_2,mn_2,mn_2+q)^-$  such that   
$$ 
\res_{s=-\beta} I_-(f,\beta_1,\beta_2,\beta_3)(s)=0, \qquad  
\qquad \forall(\beta_1,\beta_2)\in \bz_{\geq1},\beta_3\in\bz_{\geq 0}. 
$$
\end{clm}
Assuming this Claim (to be proved in the Appendix~\ref{sec:appendix}), the result follows the arguments of the end of the proof
of Proposition~\ref{prop:B12}.
\end{proof}


\section{\texorpdfstring{Bounds for $ \dim  \tilde{H}_0^{''}/ H_0^{''}$}{Bounds for dim ~H0''/H0''}}
\label{sec:bounds}

\begin{prop}\label{prop:cota1}
Let $f$ be an irreducible germ of plane curve whose characteristic sequence is $CS_{n_1,m}^{n_2,q}$
satisfying~\eqref{eq:simple_roots}.
Define the integers $\tilde{q}$ and $h$ by  $m=\tilde{q}n_1+r_m ,0<r_m<n_1$, $q=hn_2+r_q,0\leq h, 0<r_q<n_2$. Then
$$ \dim  \tilde{H}_0^{''}/ H_0^{''}\leq \frac{\mu}{2}-n_2(m+n_1)-q+\tilde{q}+h+4$$
and generically $ \dim  \tilde{H}_0^{''}/ H_0^{''}= \frac{\mu}{2}-n_2(m+n_1)-q+\tilde{q}+h+4$
\end{prop}

\begin{proof}
One can compute the dimension $ \dim  \tilde{H}_0^{''}/ H_0^{''}$ using Proposition~\ref{sp-b-exp:hs},
that  is  $\dim  \tilde{H}_0^{''}/ H_0^{''}= \sum_{i=1}^\mu \alpha_i-\sum_{i=1}^\mu \beta_i$,  the $\alpha_i$ being the spectral numbers, 
which are constant in the $\mu$-constant stratum $S_{\mu}$, and the $\beta_i$ being the $b$-exponents of $f$, which, under the conditions of 
the Proposition, equal the roots of the local Bernstein-Sato polynomial $b_{f}(-s)$.

The fact that the upper bound of $ \dim  \tilde{H}_0^{''}/ H_0^{''}$ is attained generically is a consequence of the proof of Yano's 
conjecture in \cite{ACLM-Yano2}, see also Proposition~\ref{semi}, so that the roots of $b_{\mu,gen}(-s)$ 
is the set $B_1\cup B_2.$

We start by adding the spectral numbers $\alpha_i \in \spec(f)=A_1\cup A_1^\perp\cup A_2\cup A_2^\perp$, 
see~\eqref{eq:union}. Using the symmetry of the spectral  numbers, i.e. $\alpha_i+\alpha_{\mu-(i-1)}=2$,  
for each characteristic pair, $k=1,2,$
one has:  
$$
\sum_{\alpha_i\in A_k\cup A_k^\perp} \alpha_i=2|A_k|.
$$
Using  Saito's result~\cite{MS00},
$$
|A_1|=\frac{n_2(m-1)(n_1-1)}{2}\quad\eqref{eq:a1},\quad |A_2|=\frac{ (n_2-1)(n_1n_2m+q-1)}{2},\quad
\eqref{eq:a2}.
$$
Next we compute the sum of the roots of $b_{\mu,gen}(-s)$ which is the same as the sum of the elements of $B_1\cup B_2$. Let us start with
\begin{equation}
B_1:=\left\{\sigma=\frac{m+n_1+k}{mn_1n_2}: 0\leq k< mn_1n_2, 
\text{ and }n_2 m\sigma,n_2 n_1 \sigma\notin\bz\right\}.
\end{equation}
Since $\gcd(n_1,m)=1$ and define  $N_1:=mn_1n_2$ and for the first characteristic exponent one has
\begin{gather*}
\sum_{\beta_i \in B_1} \beta_i=
\sum_{k=0}^{mn_1n_2-1}\frac{m+n_1+k}{mn_1n_2}-
\sum_{\underset{n_1+k\in m\bz}{0\leq k<m n_1 n_2} }\frac{m+n_1+k}{mn_1n_2}\\
-\sum_{\underset{m+k\in n_1\bz}{0\leq k<m n_1 n_2} }\frac{m+n_1+k}{N_1}+
\sum_{\underset{n_1+m+k\in m n_1\bz}{0\leq k<m n_1 n_2} }\frac{m+n_1+k}{N_1}
\end{gather*}
Using  
\begin{equation*}\label{sum_n_m}
 \sum_{j=n}^m j =\frac{m(m+1)}{2}-\frac{n(n-1)}{2},
\end{equation*}
the first summand is
\begin{gather*}
m+n_1+\frac{1}{m n_1 n_2}\binom{m n_1 n_2}{2}
=n_1+m+\frac{mn_1n_2-1}{2}.
\end{gather*}
For the second summand, we look for $0\leq k<m n_1 n_2$ such that if 
$m+n_1+k=ms$ for some $s\in\bz$. The minimum of such $s$ is 
$\left\lceil\frac{m+n_1}{m}\right\rceil=2$, while the maximum is
$$
\left\lfloor\frac{m+n_1+m n_1 n_2-1}{m}\right\rfloor=n_1 n_2+1
$$
Hence, the second term is
$$
-\sum_{s=2}^{n_1 n_2+1}\frac{s}{n_1 n_2}=-\frac{(n_1 n_2+2)(n_1 n_2+1)}{2n_1 n_2}+\frac{1}{n_1 n_2}=
-\frac{n_1 n_2+3}{2}.
$$
For the third term, we proceed in the same way; the extremities are
$$
\left\lceil\frac{m+n_1}{n_1}\right\rceil=2+\tilde{q},\quad
\left\lfloor\frac{m+n_1+m n_1 n_2-1}{n_1}\right\rfloor=1+m n_2+\tilde{q};
$$
the third term is
$$
-\!\!\!\!\!\!\!\sum_{s=2+\tilde{q}}^{m n_2+\tilde{q}+1}\frac{s}{m n_2}=
-\frac{(m n_2+\tilde{q}+2)(m n_2+\tilde{q}+1)-(\tilde{q}+2)(\tilde{q}+1)}{2m n_2}=
-\frac{m n_2+2\tilde{q}+3}{2}.
$$
For the fourth term the extremities are
$$
\left\lceil\frac{m+n_1}{m n_1}\right\rceil=1,\quad
\left\lfloor\frac{m+n_1+m n_1 n_2-1}{m n_1}\right\rfloor=n_2;
$$
the fourth term is
$$
\sum_{s=1}^{n_2}\frac{s}{n_2}=\frac{n_2+1}{2}.
$$
As a consequence,
\begin{gather*}
\sum_{\beta_i \in B_1} \beta_i=
n_1+m-\tilde{q}-3+n_2\frac{(m-1)(n_1-1)}{2}
\end{gather*}
For
\begin{equation}
B_2:=\left\{\left.\sigma=\frac{(m+n_1)n_2+q+k}{n_2\underbrace{(mn_1n_2+q)}_{D}}\right\vert0\leq k<n_2D
\text{ and }n_2\sigma,D\sigma\notin\mathbb{Z}\right\}.
\end{equation}
we decompose $\sum_{\beta_i \in B_2} \beta_i$ again in four terms.
For the first one, we have
$$
(m+n_1)n_2+q+\frac{n_2(mn_1n_2+q)-1}{2}.
$$
For the next terms we proceed as in the case of the first exponent. The limits of
the second term are:
\begin{gather*}
\left\lceil\frac{(m+n_1)n_2+q}{n_2}\right\rceil=m+n_1+h+1,\\
\left\lfloor\frac{(m+n_1)n_2+q+n_2(mn_1n_2+q)-1}{n_2}\right\rfloor=mn_1n_2+q+m+n_1+h;
\end{gather*}
the second term is
\begin{gather*}
-\sum_{s=m+n_1+h+1}^{mn_1n_2+q+m+n_1+h}\frac{s}{mn_1n_2+q}=
-\frac{mn_1n_2+q+2(m+n_1+h)+1}{2}
\end{gather*}
The limits of
the third term are:
\begin{gather*}
\left\lceil\frac{(m+n_1)n_2+q}{mn_1n_2+q}\right\rceil=1,\\
\left\lfloor\frac{(m+n_1)n_2+q+n_2(mn_1n_2+q)-1}{mn_1n_2+q}\right\rfloor=n_2;
\end{gather*}
the third term is
\begin{gather*}
-\sum_{s=1}^{n_2}\frac{s}{n_2}=
-\frac{n_2+1}{2}.
\end{gather*}
Finally, the limits for the fourth term are
\begin{gather*}
\left\lceil\frac{(m+n_1)n_2+q}{n_2(mn_1n_2+q)}\right\rceil=1,\\
\left\lfloor\frac{(m+n_1)n_2+q+n_2(mn_1n_2+q)-1}{n_2(mn_1n_2+q)}\right\rfloor=1;
\end{gather*}
the fourth term is $1$. Then,
\begin{gather*}
\sum_{\beta_i \in B_2} \beta_i=
(m+n_1)n_2+q-(m+n_1+h)-1+\frac{(n_2-1)(mn_1n_2+q-1)}{2}.
\end{gather*}
Recall that
$$
\mu=n_2(n_1-1)(m-1)+(n_2-1)(m n_1 n_2+q-1).
$$
The sum of the exponents is
$$
\frac{\mu}{2}+(m+n_1)n_2+q-\tilde{q}-h-4
$$
while the sum of the spectral numbers is $\mu$.
Then, its difference
is
$$
\frac{\mu}{2}-(m+n_1)n_2-q+\tilde{q}+h+4
$$
as stated.
\end{proof}

\begin{prop}\label{prop:cota2}
Let $f$ be an irreducible germ of plane curve whose characteristic sequence is $CS_{n_1,m}^{n_2,q}$
satisfying~\eqref{eq:simple_roots}. 
Then the following lower bound for $\dim  \tilde{H}_0^{''}/ H_0^{''}$ is obtained: 
$$ (n_2-1)(m-1)(n_1-1) \leq \dim  \tilde{H}_0^{''}/ H_0^{''}$$
\end{prop}

\begin{proof}
We are going to count some spectral numbers $\alpha\in A_1^\perp\cup A_2^\perp$ 
such that  $\alpha-1\in B_{11}\cup B_{21}$. This number is a lower bound for $\dim  \tilde{H}_0^{''}/ H_0^{''}$.  

Let $\alpha\in A_1^\perp$;
we have 
$$
\alpha -1=1-\frac{1}{n_2}\left(\frac{i}{n_1}+\frac{j}{m}+r\right),
$$ 
with  $\frac{i}{n_1}+\frac{j}{m}<1,r<n_2$. Assume that $r<n_2-1$:
$$
\alpha -1=\frac{n_1m(n_2-r-1)+n_1m-mi-n_1j}{n_1n_2m}\in B_{11}
$$
since the numerator is in~$\Gamma$. Hence, we have found at least
$(n_2-1)\frac{(m-1)(n_1-1)}{2}$ such numbers.

Let  $\alpha\in A_2^\perp$:
$$
\alpha-1=1-\left(\frac{i}{n_2}+\frac{j}{n_1n_2m+q}\right)=\frac{(n_2-i-1)(n_1n_2m+q)+q+n_2(n_1m-j)}{n_2(n_1n_2m+q)}.
$$
A necessary (and by the way sufficient condition) for
$\alpha-1 \in B_{21}$ is the existence of
$\beta_1,\beta_2\in\bz_{\geq 1}$ such that $mn_1-j=m\beta_1+n_1\beta_2$. 
We found another set of $(n_2-1)\frac{(m-1)(n_1-1)}{2}$ such numbers.
\end{proof}

\begin{remark}\label{rem:cotaHS}
In~\cite[Proposition~3.5]{HS99}, another bound for $\dim \tilde{H}_0^{''}/ H_0^{''}$ is given
which depends on the analytical properties of the germ, namely,
$\mu-\tau\leq \dim \tilde{H}_0^{''}/ H_0^{''}$ which yields the following
bound for the Tjurina number:
$$\tau \geq \frac{\mu}{2}+n_2(m+n_1)+q-\tilde{q}-h-4$$
\end{remark}

\section{Examples}

\begin{ejm}
We consider the case studied by Tamaki Yano, that is the characteristic sequence $(4,6,6+q)$,
i.e. $n_1=2$, $m=3$, $n_2=2$ and $q=q$. 
T.~Yano in 1983 claimed the proof of his conjecture in this case, but referred to a non published article. 
The set of spectral numbers is
\begin{gather*}
\spec(f)=
\overbrace{\left\{\frac{5}{12},\frac{11}{12}\right\}}^{A_1}
\cup
\overbrace{\left\{\frac{19}{12},\frac{13}{12}\right\}}^{A_1^\perp}
\cup\\
\overbrace{\left\{\left.\frac{12+q+2 j}{2(12+q)}\right|0<j\leq 6+\left\lfloor\frac{q}{2}\right\rfloor\right\}}^{A_2}
\cup
\overbrace{\left\{\left.\frac{3(12+q)-2 j}{2(12+q)}\right|0<j\leq 6+\left\lfloor\frac{q}{2}\right\rfloor\right\}}^{A_2^\perp};
\end{gather*}
it is not hard to see that
$$
A_2\cup A_2^\perp=\left\{\frac{14+q}{2(12+q)},\frac{16+q}{2(12+q)},\dots,\frac{34+3q}{2(12+q)}\right\}
$$ 
is a gap-free arithmetic sequence with step~$\frac{1}{12+q}$.
The set of spectral numbers $\alpha$ such that $\alpha-1\geq5/12$ is
$$
\spec(f)\setminus \cA = 
\left\{\frac{19}{12}\right\}\cup 
\left\{\left.\frac{12+q+2 j}{2(12+q)}\right|\left\lceil\frac{11(12+q)}{12}\right\rceil\leq j< 12+q\right\}
.
$$
Recall that we cannot ensure for these spectral numbers to be exponents. 
We also have
\begin{gather*}
B_1\!\!=\!B_{11}\!=\!\left\{\frac{5}{12},\frac{7}{12},\frac{11}{12},\frac{13}{12}\right\},\quad
B_2\!\!=\!B_{21}\!=\!\left\{\left.\frac{10+q+2k}{2(12+q)}\right|0\leq k\leq 11+q, k\neq 1\right\}.
\end{gather*}
Note also that $\frac{7}{12}$ and $\frac{10+q}{2(12+q)}$ are the only $b$-exponents
which are not spectral numbers.


%
\end{ejm}

As a consequence, we derive the following result.
\begin{thm}\label{Yano-family}
For any $f$ with characteristic sequence $(4,6,6+q)$, $B_1\cup B_2$ is its set of $b$-exponents (and also
the set of roots of the Bernstein polynomial). Moreover, $\dim  \tilde{H}_0^{''}/ H_0^{''}=2$.
\end{thm}

Note that the bounds of Propositions~\ref{prop:cota1} and~\ref{prop:cota2} are equal for any~$f$.
From Hertling-Stahlke bound of Remark~\ref{rem:cotaHS} we get that $\tau\geq\mu-2$.
The value of~$\tau$ equals $12+2q$ and it is constant in the whole stratum, see~\cite{LP}.

\begin{ejm}
We consider the case of characteristic sequence $(6,9,9+q)$ with $q=1+3k, k\in \bn$.
In this case $n_1=2$, $m=3$ and $n_2=3$. We have
\begin{gather*}
A_1=\left\{\frac{5}{18},\frac{11}{18},\frac{17}{18}\right\},\qquad
A_1^\perp=\left\{\frac{31}{18},\frac{25}{18},\frac{19}{18}\right\},\\
A_2=\left\{\left.\frac{19+3 k+3j}{3(19+3k)}\right|0< j\leq 12+2k\right\}
\cup\left\{\left.\frac{2(19+3 k)+3j}{3(19+3k)}\right|0< j\leq 6+k\right\},\\
A_2^\perp=\left\{\left.\frac{5(19+3 k)-3j}{3(19+3k)}\right|0< j\leq 12+2k\right\}
\cup\left\{\left.\frac{4(19+3 k)-3j}{3(19+3k)}\right|0< j\leq 6+k\right\}.
\end{gather*}
In fact 
$$
A_2\cup A_2^\perp=\left\{\frac{22+3k}{3(19+3k)},
\dots,\frac{73+12k}{3(19+3k)}\right\}
\cup
\left\{\frac{41+6k}{3(19+3k)},
\dots,\frac{92+15k}{3(19+3k)}\right\}
$$
is the union of two step-$\frac{1}{19+3 k}$ arithmetic sequences. 
The set $\cA$ is determined by
\begin{gather*}
\spec(f) \setminus \cA= 
\left\{\frac{25}{18},\frac{31}{18}\right\}
\cup\left\{\left.\frac{2(19+3 k)+3j}{3(19+3k)}\right|\left\lceil\frac{11(19+3k)}{18}\right\rceil\leq j<
19+3k\right\}\\
\cup\left\{\left.\frac{19+3 k+3j}{3(19+3k)}\right|\left\lceil\frac{17(19+3k)}{18}\right\rceil\leq j<
19+3k\right\}.
\end{gather*}
The sets of generic $b$-exponents are
\begin{gather*}
B_1=B_{11}=\left\{\frac{5}{18},\frac{7}{18},\frac{11}{18},\frac{13}{18},\frac{17}{18},\frac{19}{18}\right\}\\
B_2=\left\{\left.\frac{16+3k+3j}{3(19+3k)}\right|\underset{j\neq 1}{0\leq j< 19+3k}\right\}\cup
\left\{\left.\frac{17+3k+3j}{3(19+3k)}\right|\underset{j\neq k+7}{0\leq j< 19+3k}\right\},\\
B_{22}=
\left\{\left.\frac{17+3k+3j}{3(19+3k)}\right|0\leq j< k+6\right\}.
\end{gather*}
Note that $\frac{7}{18},\frac{13}{18}$ are the generic $b$-exponents in $B_1$ which are not spectral numbers.
For $B_2$ this is the case for
\begin{gather}\label{eq:ej2_no_sp}
\left\{\frac{16+3k}{3(19+3k)}\right\}
\cup\left\{\left.\frac{17+3k+3j}{3(19+3k)}\right|0\leq j< k+7\right\}.
\end{gather}
In particular, generically $\dim  \tilde{H}_0^{''}/ H_0^{''}=10+k$.
%
Among them, note that $\frac{7}{18},\frac{13}{18}\in B_{11}$ and 
$$
\frac{16+3k}{3(19+3k)},
\frac{35+6k}{3(19+3k)}\in B_{21};
$$
moreover they are the only common roots which are not spectral numbers, hence
$\dim  \tilde{H}_0^{''}/ H_0^{''}\geq 4$. We do not know if the equality can be reached.

The elements of $\spec(f)\setminus \cA$ that can jump to give generic roots of the Bernstein polynomial are 
\begin{gather*}
\left\{\frac{25}{18},\frac{31}{18}\right\}
\cup\left\{\left.\frac{2(19+3 k)+3j}{3(19+3k)}\right|\left\lceil\frac{11(19+3k)}{18}\right\rceil\leq j<
19+3k\right\}
\cup\left\{\frac{73+12 k}{3(19+3k)}\right\}.
\end{gather*}


Consider 
$$
f_t:=(x^2-y^3)^3+x^{5+k}y^2+t(x^2-y^3)x^5y^{k-1}
$$
where $t$ is chosen such that such that $f_t$ is of type $(6,9,10+3k)^-$.
Let $\beta=\frac{3(3\beta_1+2\beta_2)+3k+2}{3(19+3k)}\in B_2$ and in~\eqref{eq:ej2_no_sp}, i.e.
\begin{gather*}
\frac{3(3\beta_1+2\beta_2)+3k+2}{3(19+3k)}=\frac{17+3 k+3j}{3(19+3k)}\Leftrightarrow
3\beta_1+2\beta_2=5+j.
\end{gather*}
Since we need $\beta_1,\beta_2\geq 1$, all the cases are included but $j=1$.
We are going to prove $-\beta$ is a root of $f_t$ if $t\neq 0$, and as a consequence,
for $t\neq 0$ we have $\dim  \tilde{H}_0^{''}/ H_0^{''}\geq 9+k$.

We consider the polynomials $\tilde{f}$, $\tilde{\tilde{f}}$, $\hat{f}$, $\hat{\hat{f}}$, $f_1$, $f_2$:
\begin{gather*}
\tilde{f}(x,y)=(x^6-y^6)^3+x^{15+3k}y^4+t(x^6-y^6)x^{15}y^{2k-2},\\
\tilde{\tilde{f}}(x,y)=(1-y^6)^3+x^{1+3k}y^4+t(1-y^6)x^{1+2k}y^{2k-2},\\
\hat{f}(x,y)=(1-(1-y)^6)^3+x^{1+3k}(1-y)^4+t(1-(1-y)^6)x^{1+2k}(1-y)^{2k-2}=\\
y^3(6^3+\dots)+x^{1+3k}(1-\dots)+t x^{1+2k} y(6+\dots),\\
\hat{\hat{f}}(x,y)=y^{3(1+3k)}(6^3+\dots)+x^{3(1+3k)}(1-\dots)+t x^{3(1+2k)} y^{1+3 k}(6+\dots),\\
f_1(x,y)=y^{3(1+3k)}(6^3+\dots)+(1-\dots)+t x y^{1+3 k}(6+\dots),\\
f_2(x,y)=(6^3+\dots)+x^{3(1+3k)}(1-\dots)+t x^{3(1+2k)} y(6+\dots).
\end{gather*}
%
We have
$$ \res_{s=-\beta} I_{-}(f,\beta_1,\beta_2,0)(s)=\frac{1}{3(1+3k)}
(G_{h^1_{1,-\beta,x}}(1+3k)+G_{h^2_{1,-\beta,y}}(2(-18\beta+3\beta_1+2\beta_2)),
$$
where 
\begin{gather*}
h^1_{1,-\beta,x}(y)=\frac{\partial f_1^{-\beta}}{\partial x}(0,y)=
-6\beta ty^{1+3k}(6^3y^{3(1+3k)}+1)^{-\beta-1},\\
h^2_{1,-\beta,y}(x)=\frac{\partial f_2^{-\beta}}{\partial y}(x,0)=
-6\beta tx^{3(1+2k)}(6^3+x^{(1+3k)3})^{-\beta-1}.
\end{gather*}
%
%
If $t$ is algebraic (and $t\neq 0$), the above residue is transcendental. Hence,
we deduce that these values are roots of the Bernstein polynomial for these values of $t$. Moreover,
since the Bernstein-polynomial stratification is algebraic, we deduce that this is the case for $t\neq 0$.
Note that in this case, for $k=0$, and for random values of $t$, the Tjurina number equals~$40$,
while for $t=0$, the value is~$41$. Hence the Tjurina number is not constant in the $\mu$-constant stratum. 
\end{ejm}

\begin{ejm}
Consider the characteristic sequence $(8,10, 10+q)$, where $(q,2)=1,(q,5)=1$.
In this case $n_1=4$, $m=5$ and $n_2=2$.
We have 
\begin{gather*}
A_1=\left\{\frac{9}{40},\frac{13}{40},\frac{14}{40},\frac{17}{40},
\frac{18}{40},\frac{19}{40},\frac{29}{40},\frac{33}{40},\frac{34}{40},
\frac{37}{40},\frac{38}{40},\frac{39}{40}\right\}\\
A_1^\perp=\left\{\frac{71}{40},\frac{67}{40},\frac{66}{40},\frac{63}{40},
\frac{62}{40},\frac{61}{40},\frac{51}{40},\frac{47}{40},\frac{46}{40},
\frac{43}{40},\frac{42}{40},\frac{41}{40}\right\}\\
A_2=\left\{\left.\frac{40+q+2j}{2(40+q)}\right|0<j\leq20+\left\lfloor\frac{q}{2}\right\rfloor\right\}\\
A_2^\perp=\left\{\left.\frac{3(40+q)-2j}{2(40+q)}\right|0<j\leq20+\left\lfloor\frac{q}{2}\right\rfloor\right\}
\end{gather*}
Note that 
$$
A_2\cup A_2^\perp=
\left\{\frac{42+q}{2(40+q)},\frac{44+q}{2(40+q)},\dots,\frac{118+3q}{2(40+q)}\right\}\\
$$
is a step-$\frac{1}{40+q}$ arithmetic sequence.

Then $\spec(f)\setminus \cA$ is
\begin{gather*}
\left\{\frac{51}{40},\frac{61}{40},\frac{62}{40},\frac{63}{40},\frac{66}{40},\frac{67}{40},\frac{71}{40}\right\}
\cup\left\{\left.\frac{40+q+2j}{2(40+q)}\right|29+
\left\lceil\frac{29}{40}q\right\rceil\leq j\leq20+\left\lfloor\frac{q}{2}\right\rfloor\right\}.
\end{gather*}
With this data
\begin{gather*}
B_1=\left\{\left.\frac{9+4\ell}{40}\right|\underset{\ell\neq 4}{0\leq\ell\leq 8}\right\}
\cup\left\{\left.\frac{10+4\ell}{40}\right|\underset{\ell\neq 5}{1\leq\ell\leq 9}\right\}
\cup\left\{\left.\frac{11+4\ell}{40}\right|\underset{\ell\neq 1,6}{0\leq\ell\leq 9}\right\},
\end{gather*}
where $B_{12}=\{\frac{11}{40}\}$, and
\begin{gather*}
B_2=\left\{\left.\frac{18+q+2\ell}{2(40+q)}\right|0\leq\ell<40+q,\ell\neq 11\right\},\\
B_{22}=\left\{\frac{20+q}{2(40+q)},\frac{22+q}{2(40+q)},
\frac{24+q}{2(40+q)},\frac{30+q}{2(40+q)},\frac{32+q}{2(40+q)}\right\}.
\end{gather*}
%
 To get the Bernstein polynomial for any function with characteristic 
sequence $(8,10, 10+q)$, we only have to check for the $6$ elements
of $B_{12}\cup B_{22}$ if they are roots (recall that the Milnor number is $63+q$).
Let us study the generic $b$-exponents which are not spectral numbers:
\begin{gather*}
\left\{\frac{21}{40},\frac{22}{40},\frac{26}{40},\frac{11}{40},\frac{23}{40},\frac{27}{40},\frac{31}{40}\right\}
\cup
\left\{\left.\frac{18+q+2\ell}{2(40+q)}\right|0\leq\ell\leq 10\right\}
\supset B_{12}\cup B_{22}.
\end{gather*}
We get $12\leq  \dim  \tilde{H}_0^{''}/ H_0^{''}\leq 18$.
We shall show,
for $q=7$, 
that there exists an $f$ with characteristic sequence $(8,10,10+q)$ such that  $12= \dim  \tilde{H}_0^{''}/ H_0^{''}$.
Consider 
$$
f_{\pm}(x,y)=(x^4\pm y^5)^2+x^7y^3+tx^6y^6.
$$
We assume that $f_\pm$ has type  $(8,10,10+q)^{\pm}$.
Let 
\begin{gather*}
\tilde{f}_+(x,y)=f_+(x^4,y^5)=(x^{20}+ y^{20})^2+x^{35}y^{12}+tx^{30}y^{24}\\
f_{+1}(x,y)=(1+y^{20})^2+x^7y^{12}+tx^{14}y^{24}\\
f_{+2}(x,y)=(x^{20}+1)^2+x^{35}y^{7}+tx^{30}y^{14}
\end{gather*}
Let $\beta=\frac{11}{40}=\frac{(5+4)+2}{40}$.
The residue is 
$$
\res _{s=-\beta}I(f_+,1,1)(s)=\frac{1}{80}(G_{h^1_{2,-\beta,x}}(4)+G_{h^2_{2,-\beta,y}}(5)),
$$
where 
$$h^1_{2,-\beta,x}(y)=\frac{\partial^{2}f_1^{-\beta}}{\partial x^{2}}(0,y)=0,\quad
h^2_{2,-\beta,y}(x)=\frac{\partial^{2}f_2^{-\beta}}{\partial y^{2}}(x,0)=0.$$
Then $\res _{s=-\beta}I(f_+,1,1)(s)=0$;
moreover, with the same ideas as in the proof of Claim~\ref{claim1}
we have that $\forall (\beta_1,\beta_2), \res _{s=-\beta}I(f_+,\beta_1,\beta_2)(s)=0$.

Consider now $\beta =\frac{51}{40}= \frac{5\cdot 3+4\cdot 9}{40}$. 
We know that $I(f_+,3,9)(s) $ 
has a pole for $s=-\beta$ with transcendental residue. Combining the two facts, 
by Theorem~\ref{pole-integral-root},  $-\frac{51}{40}$ 
is a zero of the Bernstein polynomial of $f_+$ and $-\frac{11}{40}$ is not.
Since $f_+(x,y)=f_-(-x,-y)$, we deduce the same property for $f_-$.

It remains to study the cases in $B_{22}$, i.e., the set
$\left\{\frac{27}{94},\frac{29}{94},
\frac{31}{94},\frac{37}{94},\frac{39}{94}\right\}$, with $k=2,4,6,12,14$. 
Since the generators of $\Gamma$ are $8,10,47$,
only the following suitable combinations of $\beta_1,\beta_2,\beta_3,\nu$ (see proof of Claim~\ref{claim2}) are
available:
\begin{center}
\begin{tabular}{|c|c|c|c|c|c|c|}\cline{1-7}
$k$ & $\beta_1$ & $\beta_2$ & $\beta_3$ & $\nu$ & $i$ & $\nu_i$\\\cline{1-7}
 2 & 1 & 1 & 0 & 2 &1 & 2\\\cline{1-7}
 4 & 1 & 1 & 0 & 4 &1 & 4\\\cline{1-7}
 6 & 1 & 1 & 0 & 6 &1 & 6\\\cline{1-7}
 12 & 1 & 1 & 0 & 12 & 1 & 12\\\cline{1-7}
 12 & 1 & 1 & 0 & 12 & 2 & 5\\\cline{1-7}
12 & 2 & 1 & 0 & 4 & 1 & 4\\\cline{1-7}
12 & 1 & 2 & 0 & 2 & 1 & 2\\\cline{1-7}
14 & 1 & 1 & 0 & 14 & 1 & 14\\\cline{1-7}
14 & 1 & 1 & 0 & 14 & 2 & 7\\\cline{1-7}
14 & 1 & 1 & 0 & 14 & 3 & 0\\\cline{1-7}
14 & 2 & 1 & 0 & 6 & 1 & 6\\\cline{1-7}
14 & 1 & 2 & 0 & 4 & 1 & 4\\\cline{1-7}
\end{tabular}
\end{center}

Let us compute the polynomials appearing
in the different steps of the process:
\begin{gather*}
f(x,y)=(x^4-y^5)^2+x^7y^3+tx^6y^6,\\
\tilde{f}(x,y)=(x^{20}-y^{20})^2+x^{35}y^{12}+tx^{30}y^{24},\\
\tilde{\tilde{f}}(x,y)=(1-y^{20})^2+x^{7}y^{12}+tx^{14}y^{24},\\
\hat{f}(x,y)=y^2H_{20}(y)^2+x^{7}(1-y)^{12}+tx^{14}(1-y)^{24},\\
\hat{\hat{f}}(x,y)=y^{14} H_{20}(y^7)^2+x^{14}(1-y^7)^{12}+tx^{28}(1-y^7)^{24},\\
f_1(x,y)=y^{14}H_{20}(x^7 y^7)^2+(1-x^7 y^7)^{12}+tx^{14}(1-x^7 y^7)^{24},\\
f_2(x,y)=H_{20}(y^7)+x^{14}(1-y^7)^{12}+tx^{28}y^{14}(1-y)^{24}.
\end{gather*}
where $yH_n(y)=(1-(1-y)^n)$, $H_n(0)=n$.
From these data it is easy to check that
\begin{gather*}
\frac{\partial^{N}f_1}{\partial x^{N}}(0,y)=
\begin{cases}
20^2 y^{14}+1,&\text{ if }N=0\\
-20160 y^{7} (1900 y^{14} + 3),&\text{ if }N=7\\
87178291200(81700 y^{28} + 66 y^{14} + t),&\text{ if }N=14\\
0&\text{ if }\frac{N}{7}\notin\bz.
\end{cases}
\end{gather*}
and
\begin{gather*}
\frac{\partial^{N}f_2}{\partial y^{N}}(x,0)=
\begin{cases}
20^2 +x^{14},&\text{ if }N=0\\
-20160  (1900 + 3 x^{14}),&\text{ if }N=7\\
87178291200(81700  + 66 x^{14} + t x^{28}),&\text{ if }N=14\\
0&\text{ if }\frac{N}{7}\notin\bz.
\end{cases}
\end{gather*}
With the same ideas
\begin{gather}\label{eq:derf1}
\frac{\partial^{N}f_1^{-\beta}}{\partial x^{N}}(0,y)=
\begin{cases}
(20^2 y^{14}+1)^{-\beta},&\text{ if }N=0\\
20160\beta y^{7} (1900 y^{14} + 3)(20^2 y^{14}+1)^{-\beta-1},&\text{ if }N=7\\
\!-\!87178291200\!\beta(81700 y^{28}\!\! +\! 66 y^{14}\!\! +\! \!t)(20^2 y^{14}\!\!+\!\!1)^{-\beta-1}+&\\
697426329600\beta(\beta\!\!+\!\!1)\!y^{14} (1900 y^{14}\!\! +\!\! 3)^2(20^2 y^{14}\!\!+\!\!1)^{-\beta-2}&\text{ if }N=14\\
0&\text{ if }\frac{N}{7}\notin\bz.
\end{cases}
\end{gather}
and
\begin{gather}\label{eq:derf2}
\frac{\partial^{N}f_2^{-\beta}}{\partial y^{N}}(x,0)=
\begin{cases}
(20^2 +x^{14})^{-\beta},&\text{ if }N=0\\
20160\beta (1900  + 3x^{14})(20^2 +x^{14})^{-\beta-1},&\text{ if }N=7\\
\!-\!87178291200\!\beta\!(81700\!\! +\! 66 x^{14}\! + \!t x^{28})(20^2 \!\!+\!\!x^{14})^{-\beta-1}+&\\
697426329600\beta(\beta\!+\!1) (1900 \! +\! 3 x^{14})^2(20^2 \!+\!x^{14})^{-\beta-2}&\text{ if }N=14\\
0&\text{ if }\frac{N}{7}\notin\bz.
\end{cases}
\end{gather}
It becomes obvious that all the residues vanish for $k=2,4,6,12$.
The residue also vanishes for $k=14$ and $(\beta_1,\beta_2)=(2,1),(1,2)$.
Let us study the case $k=14$, i.e., $\beta=\frac{39}{94}$, with $(\beta_1,\beta_2)=(1,1)$.
Note that for $\beta_2=1$,
\begin{gather*}
h_4(y)=
(1-y)^3=1-3 y+3 y^2-y^3.
\end{gather*}
Hence,
\begin{gather*}
14\res _{s=-\beta}I_-(f,1,1,0)(s)=
\frac{1}{14!}\left(G_{h^1_{14,-\beta,x}}(7)+G_{h^2_{14,-\beta,y}}\left(-\frac{714}{47}\right)\right)\\-
\frac{3}{7!}\left(G_{h^1_{7,-\beta,x}}(14)+G_{h^2_{7,-\beta,y}}\left(-\frac{714}{47}\right)\right)+
3\left(G_{h^1_{0,-\beta,x}}(21)+G_{h^2_{0,-\beta,y}}\left(-\frac{714}{47}\right)\right)
\end{gather*}
We find the values of $h_1(k,-\beta,x)$ and $h_2(k,-\beta,y)$ in~\eqref{eq:derf1} and~\eqref{eq:derf2}.
We can prove that the pole at $s=-\beta$ of $I_-(f,1,1)(s)$ is a polynomial of degree $1$ in $t$
and hence there is a value of~$t$ for which the residue vanishes. Moreover
\begin{gather*}
\res _{s=-\beta}I_-(f,1,1,0)(s)=
\frac{\left(136 \, t - 63\right)}{447440} {\boldsymbol{B}}\left(-\frac{4}{47}, \frac{1}{2}\right).
\end{gather*}
In particular, none of the above elements are roots of the Bernstein polynomial of $f_-$ for $t=\frac{63}{136}$.
For $t\neq\frac{63}{136}$, $-\frac{39}{94}$ is such a root but not for $t=\frac{63}{136}$. This can be 
confirmed using \texttt{checkRoot} of~\cite{LMM:12} in \texttt{Singular}~\cite{DGPS}, inside~\cite{sage}.
Moreover, it can be proved that for general~$t$ (including $\frac{63}{136}$) the Tjurina number equals
the expected value for Hertling-Stahlke bound, i.e., $58$; using~\cite{LP}
the values of Tjurina number are constant in these $\mu$-constant strata, namely
they equal $51+q$. In particular, Bernstein and Tjurina stratifications
do not coincide. 

\end{ejm}

\appendix
\section{Technical proofs}\label{sec:appendix}
\begin{proof}[Proof of {\rm Claim~\ref{claim1}}]
Let us recall that $\beta=\frac{m+n_1+k}{m n_1 n_2}$ and
$k\notin\Gamma_1$.

Let $(\beta_1,\beta_2)\in\mathbb{Z}^2_{\geq 1}$. If $\beta_1 m+\beta_2 n_1>m+n_1+k$,
the greatest pole  of
$I(f,\beta_{1},\beta_{2})(s)$
is smaller than $-\beta$ and the statement holds trivially for any~$f$.

We want to fix our attention on the couples 
$(\beta_1,\beta_2)\in\mathbb{Z}^2_{\geq 1}$ such that $\beta_1 m+\beta_2 n_1\leq m+n_1+k$.
There is a finite number of such couples which will be characterized in the following paragraphs.

Since
$k\notin\Gamma_1$, and from its properties, 
we know that $k\leq m n_1-m-n_1$. We write
\begin{equation}
k=m i_0+ n_1 j_0- m n_1,\quad
1\leq i_0< n_1,\quad 1\leq j_0 <  m.
\end{equation}
Moreover the pair 
of positive integers $(i_0,j_0)$ is unique. Let us assume the existence of another solution
$(i_1,j_1)$, such that $i_1>i_0$; then $i_1=i_0+n_1 v$, $v\in\mathbb{Z}_{>0}$, i.e., $i_1>n_1$, 
leading to a contradiction.

We are going to prove also that $\beta_1\leq i_0$ and $\beta_2\leq j_0$. Let us assume that
$\beta_1>i_0$. Then
\begin{gather*}
(i_0+1) m+\beta_2 n_1<\beta_1 m+\beta_2 n_1\leq m+n_1+k=m (i_0+1)+ n_1 (j_0+1)- m n_1\\
\Longrightarrow
\beta_2+m\leq j_0+1<m+1,
\end{gather*}
which is a contradiction. We are going to enumerate these couples $(\beta_1,\beta_2)$.

Let us define $\ell_{i j}:= m i +n_1 j -n_1m $ and consider
$$
\{\ell_{ij}\mid \ell_{ij}\geq 1,1\leq i\leq i_0,1\leq j\leq j_0\}=\{\ell_1,\dots,\ell_r\},\quad
\ell_1<\dots<\ell_r=k.
$$ 
Let
\begin{equation}\label{eq:lp}
\ell_p= m i_p +n_1 j_p-m n_1,\quad i_p,j_p \text{ well-defined},
\text{ for }1\leq p\leq  r.
\end{equation}
For each $p$ we can write
$$
\beta= \frac{m\overbrace{(1+i_0-i_p)}^{\beta_{1p}}+ n_1 \overbrace{(1+j_0-j_p )}^{\beta_{2p}}
+ \ell_p}{n_1n_2 m};
$$
note that $\beta_{1r}=\beta_{2r}=1$ and $1\leq \beta_{1p}\leq i_0$, $1\leq \beta_{2p}\leq j_0$.
It is easy to prove that 
$\{(\beta_{1p},\beta_{2p})\mid 1\leq p\leq r\}=
\{(\beta_1,\beta_2)\in\mathbb{Z}_{\geq 1}\times\mathbb{Z}_{\geq 1}\mid\beta_1 m+\beta_2 n_1\leq m+n_1+k\}$.
These $r$~pairs are exactly the ones for which we need to prove the statement.

Define 
$$
f_\mathbf{t}(x,y): = \left( x^{n_1} + y ^{m} + \sum _{p=1}^r t_p x^{i_p} y^{j_p} \right)^{n_2} +x^a y^b.
$$
with $ma+n_1b=q+mn_1n_2$  and $\mathbf{t}=(t_1,\ldots,t_r)\in \br^r$ such that $f_\mathbf{t}$ is of type $(n_1n_2,mn_2,mn_2+q)^+$.
 By Proposition~\ref{prop:newton1} one has 

$$\tilde{f}_\mathbf{t}(x,y)=(x^{mn_1}+y^{mn_1}+\sum _{p=1}^r t_p x^{mi_p} y^{n_1j_p} )^{n_2} +x^{ma} y^{n_1b}.$$
$$f_{\mathbf{t},1}(x,y)=(1+y^{mn_1}+\sum _{p=1}^r t_p x^{\ell_p} y^{n_1j_p} )^{n_2} +x^{q} y^{n_1b}.$$
$$f_{\mathbf{t},2}(x,y)=(x^{mn_1}+1+\sum _{p=1}^r t_p x^{mi_p} y^{\ell_p} )^{n_2} +x^{ma} y^{q}.$$

Let us fix~$p\in\{1,\dots,r\}$. To compute the residue of $I_+(f_\mathbf{t},\beta_{1p},\beta_{2p})(s)$ at 
$s=-\beta$ we apply equation~(\ref{residue-B1}) and we get
$$ 
\rho_p:=\res_{s=-\beta} I_+(f_\mathbf{t},\beta_{1p},\beta_{2p})(s)=
\frac{1}{\ell_p!mn_1n_2}\left(G_{h^1_{\ell_p,-\beta,x}}(n_1\beta_{2p}
)+G_{h^2_{\ell_p,-\beta,y}}
(m\beta_{1p})\right)
$$
where 
$$
h^1_{\ell_p,-\beta,x}(y)=\frac{\partial^{\ell_p}f_{\mathbf{t},1}^{-\beta}}{\partial x^{\ell_p}}(0,y) \textrm{ and }\,  
h^2_{\ell_p,-\beta,y}(x)=\frac{\partial^{\ell_p}f_{\mathbf{t},2}^{-\beta}}{\partial y^{\ell_p}}(x,0),
$$
recall that $G_f(s)$ is meromorphic continuation of 
$
\int_0^1 f(t)t^s\frac{dt}{t}.
$

We have
$$
\frac{\partial ^{\ell_p}f_{\mathbf{t},1}^{-\beta}}{\partial x^{\ell_p}}(0,y)=
\sum_{V=(u_w)_{w=1}^{|V|}\in\mathcal{P}(\ell_p)
} D_V
\left(\prod_{w=1}^{|V|}
\frac{\partial ^{u_w}f_{\mathbf{t},1}}{\partial x^{u_w}}(0,y)\right)
(1+y^{mn_1})^{-n_2(\beta+\vert V\vert)}
$$
where 
\begin{equation}\label{eq:plp}
\mathcal{P}(\ell_p)=\left\{V=(u_w)_{w=1}^{|V|}\Bigg| \sum_{w=1}^{|V|} u_w=\ell_p, u_1\leq\dots\leq u_{|V|}\right\},
\end{equation}
and $D_V\in \mathbb{Q}$. In the same way,
$$
\frac{\partial ^{\ell_p}f_{\mathbf{t},2}^{-\beta}}{\partial y^{\ell_p}}(x,0)=
\sum_{V=(u_w)_{w=1}^{|V|}\in\mathcal{P}(\ell_p)
} D_V
\left(\prod_{w=1}^{|V|}
\frac{\partial ^{u_w}f_{\mathbf{t},2}}{\partial y^{u_w}}(x,0)\right)
(1+x^{mn_1})^{-n_2(\beta+\vert V\vert)}.
$$

Let us study now the  $u^{\text{th}}$ $x$-derivative of $f_{\mathbf{t},1}$ evaluated at $(0,y)$, i.e.,
we need to look for the monomials of the type $x^u y^j$, for any~$j$.
Hence,
$$
\frac{\partial ^{u}f_{\mathbf{t},1}}{\partial x^{u}}(0,y)=\delta_q^u q!y^{n_1b}+
\sum_{K=(k_h)_{h=1}^r,u=\sum k_h \ell_h} 
C_K\left(\prod_{h=1}^r t_h^{k_h}y^{n_1 k_h j_h}\right)(1+y^{mn_1})^{n_2-\sum_{h=1}^r k_h},
$$
for some $C_K\in\mathbb{Q}$, where $\delta_q^u$ is the Kronecker symbol.
A similar formula holds for derivatives with respect to~$y$:
$$
\frac{\partial ^{u}f_{\mathbf{t},2}}{\partial y^{u}}(x,0)=\delta_q^u q!x^{m a}+
\sum_{K=(k_h)_{h=1}^r,u=\sum k_h \ell_h} 
C_K\left(\prod_{h=1}^r t_h^{k_h} x^{m k_h i_h}\right)(1+x^{mn_1})^{n_2-\sum_{h=1}^r k_h},
$$
Let us  compute the residue $\rho_p$. It is a linear combination with coefficients in $\mathbb{Q}$ of
terms depending on couples $(V,(K_w)_w)$ where $V=(u_w)\in\mathcal{P}(\ell_p)$ and
for each $w\in\{1,\dots,|V|\}$, $K_w=(k_{h,w})_{h=1}^r$ satisfies either
%
\begin{equation}\label{eq:uw}
\sum_{h=1}^r k_{h,w}\ell_h=u_w;
\end{equation}
or the term involved is $y^{n_1b}$ or $x^{m a}$, i.e.,
\begin{equation}\label{eq:uwq}
u_w=q,\quad k_{h,w}=0;
\end{equation}
let $r_V$ be the number of terms of this type for $V$
then, the term is obtained as
\begin{gather*}
\!\int _0^1\!\!\prod_{w=1}^{|V|}\!\!
\left(
\prod_{h=1}^r\! t_h^{k_{h,w}}\!y^{n_1 k_{h,w} j_h}\!\!
\right)\!\!
y^{r_V n_1 b}\!
(1\!+y^{mn_1})^{\sum_{w}(n_2-\sum_{h=1}^r k_{h,w})-n_2(\beta+\vert V\vert+r_V)}y^{n_1(1+j_0-j_p)}\frac{dy}{y}+\!\\
\!\int_0^1\!\!\prod_{w=1}^{|V|}\!\!
\left(
\prod_{h=1}^r\! t_h^{k_{h,w}}\!x^{m k_{h,w} i_h}\!\!
\right)\!
x^{r_V m a}
(1+x^{mn_1})^{\sum_{w}(n_2-\sum_{h=1}^r k_{h,w})-n_2(\beta+\vert V\vert+r_V)}x^{m(1+i_0-i_p)}\frac{dx}{x}.
\end{gather*}
This is a monomial in $t_1,\dots,t_r$, namely,
$$
\prod_{h=1}^r
t_h^{\sum_w k_{h,w}}
$$
whose coefficient is
\begin{gather*}
\rho_{V,(K_w)}:=\int _0^1
y^{n_1(\sum_{w,h} k_{h,w} j_h+1+j_0-j_p+r_V b)}
(1+y^{mn_1})^{-n_2(\beta+r_V)-\sum_{w,h} k_{h,w}}\frac{dy}{y}+\\
\int _0^1
x^{m(\sum_{w,h} k_{h,w} i_h+1+i_0-i_p+r_V a)}
(1+x^{mn_1})^{-n_2(\beta+r_V)-\sum_{w,h} k_{h,w}}\frac{dx}{x}=\\
G_{(1+y^{m n_1})^\alpha}\left(n_1\left(\sum_{w,h} k_{h,w} j_h+1+j_0-j_p+r_V b\right)\right)+\\
G_{(1+x^{m n_1})^\alpha}\left(m\left(\sum_{w,h} k_{h,w} i_h+1+i_0-i_p+r_V a\right)\right)
\end{gather*}
where
$$
\alpha:=-n_2(\beta+r_V)-\sum_{w,h} k_{h,w}.
$$
We need to compute the sum of the arguments
\begin{gather}\label{eq:suma}
\sigma:=n_1\left(\sum_{w,h} k_{h,w} j_h+1+j_0-j_p+r_V b\right)+
m\left(\sum_{w,h} k_{h,w} i_h+1+i_0-i_p+r_V a\right).
\end{gather}
From the equalities \eqref{eq:uw}, \eqref{eq:uwq} and the definition \eqref{eq:lp}, we have
\begin{gather*}
u_w=m\sum_{h=1}^r k_{h,w}i_h+n_1\sum_{h=1}^r k_{h,w}j_h-mn_1\sum_{h=1}^r k_{h,w}
\end{gather*}
if the first term is not involved and $u_w=q$ if it is; recall also that
\begin{equation}\label{eq:q}
q=ma+n_1b-mn_1n_2.
\end{equation}
Then
\begin{gather}\label{eq:mn1}
i_p m\!+\!n_1 j_p\!-\!m n_1=\ell_p=\sum_w u_w=\\ \notag
m\!\!\left(\sum_{w,h} k_{h,w} i_h+r_V a\right)\!
+\!n_1\!\!\left(\sum_{w,h} k_{h,w}j_h+r_V b\right)\!-\!mn_1\!\!\left(\sum_{w,h}  k_{h,w}+r_V n_2\right)\!.
\end{gather}
We obtain several properties from this equality. 
In particular
\begin{gather*}
\sigma=
m n_1\left(\sum_{w,h} k_{h,w}+r_V n_2\right)+
\overbrace{n_1\left(1+j_0\right)+m\left(1+i_0\right)-m n_1}^{m +n_1+k=m n_1 n_2 \beta}\\=
m n_1\left(n_2(\beta+r_V)+\sum_{w,h} k_{h,w}\right)=-m n_1\alpha.
\end{gather*}
By Lemma~\ref{beta}, we have that
$$
\!\rho_{V,(K_w)}\!=\!\!\frac{1}{m n_1}\!\boldsymbol{B}\!
\left(
\frac{\sum_{w,h} k_{h,w} j_h+1+j_0-j_p+r_V b}{m}\!,\!
\frac{\sum_{w,h} k_{h,w} i_h+1+i_0-i_p+r_V a}{n_1}\!
\right)\!.\!
$$
As another consequence from \eqref{eq:mn1}, 
we have that
\begin{gather*}
\frac {\sum_{w,h} k_{h,w}j_h-j_p+r_V b}{m},\frac {\sum_{w,h} k_{h,w}i_h-i_p+r_V a}{n_1}\in \mathbb{Z}. 
\end{gather*}
Let us prove it. Since
$\gcd(m,n_1)=1$, it is enough to show that the product of $n_1$ and the first denominator is congruent to $0\bmod m$:
\begin{gather*}
n_1\!\left(\sum_{w,h}\! k_{h,w}j_h\!+\!r_V b\!-\!j_p\!\right)\!\!=\!
m\!\!\left(\!
i_p\!-\!n_1\!-\!\sum_{w,h} k_{h,w}i_h \!-\!r_Va\!
\right)\!
+\!m n_1\left(\sum_{w,h}  k_{h,w}+r_V n_2\right).
\end{gather*}

From the properties of the beta function, $\rho_{V,(K_w)}$ is a product of a non-zero rational number and 
$\boldsymbol{B}\left(\frac{1+i_0}{n_1},\frac{1+j_0}{m}\right)$.
As a consequence
$\res_{s=-\beta} 
I_+(f_\mathbf{t},\beta_{1i},\beta_{2i})(s)$
is, up to the factor $\boldsymbol{B}\left(\frac{1+i_0}{n_1},\frac{1+j_0}{m}\right)$ a 
polynomial $Q_p$ in the $t_i$'s with coefficients in~$\mathbb{Q}$; the coefficient of $t_p$ 
does not vanish. 
The only option to have the monomial $t_p$ is when $V=(\ell_p)$ and $K=(\ell_p)$, $r_V=0$ and 
for these values
$$
\rho_{V,(K_w)}=\frac{1}{m n_1}\boldsymbol{B}\left(\frac{1+i_0}{n_1},\frac{1+j_0}{m}\right).
$$


Since $\ell_1$ is the minimum, $Q_1$ is a polynomial in $t_1$ of degree~$1$.
Then we can choose $t_1$ such that  $\res_{s=-\beta} I_+(f_\mathbf{t},\beta_{11},\beta_{21})(s)=0$, 
since this residue is independent of $t_p$, for $p>1$. 
From now on $f_{\mathbf{t}}$ is a polynomial in $t_2,\dots,t_r$, with $t_1$ fixed as above.
In the same way,
we choose $t_2$ such that
$$
\res_{s=-\beta} I_+(f_\mathbf{t},\beta_{12},\beta_{22})(s)=0,
$$
and recursively we can find $t_3,\dots,t_r$ such that
 $\res_{s=-\beta} I_+(f_\mathbf{t},\beta_{1p},\beta_{2p})(s)=0$, 
for all $1\leq p\leq r$ and all the $t$'s are in $\mathbb{Q}$.
Using Proposition~\ref{123to4}  it is easy to proof that  $f_\mathbf{t}$ is of type $(n_1n_2, mn_2,mn_2+q )^+$.
%
\end{proof}

\begin{proof}[Proof of {\rm Claim~\ref{claim2}}]
Let $(\beta'_1,\beta'_2,\beta_3)$ be as in the statement. If $n_2(\beta'_1 m+\beta'_2 n_1)+\beta_3(n_1 n_2 m+q)>n_2(m+n_1)+k$,
it is not hard to check that the statement holds trivially for any~$f$
of type $(n_1n_2,mn_2,mn_2+q)^-$.

We are going to characterize  the triples not satisfying the above inequality
and to find an $f_\beta$ satisfying the conditions of the statement.
%
Let
$$
M_\beta=
\left\{(\tilde{\beta}_1,\tilde{\beta}_2,\beta_3,\nu)\in\bz_{\geq 0}^3\times\bz_{\geq 1}\mid
k=n_2(m\tilde{\beta}_1+n_1\tilde{\beta}_2)+(mn_1n_2+q)\beta_3+\nu\right\}.
$$
It is not hard to prove the following properties:
\begin{itemize}
\item if $(\tilde{\beta}_1,\tilde{\beta}_2,\beta_3,\nu)\in M_\beta$,
then $\beta_3<n_2$;
\item if moreover $(\tilde{\beta}'_1,\tilde{\beta}'_2,\beta'_3,\nu)\in M_\beta$
then $\beta_3=\beta'_3$.
\end{itemize}

We denote by $N_\beta$ the set of $\nu$ which are the fourth coordinate of some
element of $M_\beta$ and we order $N_\beta$.
For $\nu\in N_\beta$, choose 
$\tilde{\beta}_1,\tilde{\beta}_2,\beta_3$ such that 
$(\tilde{\beta}_1,\tilde{\beta}_2,\beta_3,\nu)\in M_\beta$; if we denote
$\beta_i=\tilde{\beta}_i+1$, $i=1,2$, we have:
$$
n_2(m+n_1)+k=n_2(m\beta_1+n_1 \beta_2)+(mn_1n_2+q)\beta_3+\nu.
$$
Note that $\beta_3$ is determined by~$\nu$; it may not be the case for $\beta_1,\beta_2$. 
Let $\ell_{\nu}$ such that $0\leq \ell_{\nu}<n_2$,  and $a_{\nu},b_{\nu}\in \bz_{\geq 0}$ 
such that 
$$
(mn_1n_2+q)\ell_{\nu}+(ma_{\nu}+n_1b_{\nu})n_2=(mn_1n_2+q)n_2+\nu.
$$
Let 
$$
f_{\beta}(x,y)=(x^{n_1}-y^{m})^{n_2}+x^ay^b+\sum_{\nu\in N_{\beta}} t_{\nu} (x^{n_1}-y^{m})^{\ell_{\nu}}x^{a_{\nu}}y^{b_{\nu}}
$$
We choose $f_{\beta}$ of type $(n_1n_2,mn_2,mn_2+q)^-$. Let us recall
the change of variables that allows to compute the poles of the proper integrals.
Note that in this case, one can choose $g_Y=x^{n_1}-y^{m}$.
We have:
\begin{gather*}
\tilde{f}_{\beta}(x,y)=(x^{mn_1}-y^{mn_1})^{n_2}+x^{ma}y^{n_1b}+
\sum_{\nu\in N_{\beta}} t_{\nu} (x^{mn_1}-y^{mn_1})^{\ell_{\nu}}x^{ma_{\nu}}y^{n_1b_{\nu}}\\
\tilde{\tilde{f}}_{\beta}(x,y)=(1-y^{mn_1})^{n_2}+x^{q}y^{n_1b}+
\sum_{\nu\in N_{\beta}} t_{\nu} (1-y^{mn_1})^{\ell_{\nu}}x^{\frac{\nu+q(n_2-\ell_{\nu})}{n_2}}y^{n_1b_{\nu}}
\end{gather*}
\begin{gather*}
\hat{f}_{\beta}(x,y)=y^{n_2}h_1(y)+
x^qh_2(y)+\sum_{\nu\in N_{\beta}} t_{\nu}y^{\ell_{\nu}} h_{3,\nu}(y) x^{\frac{\nu+q(n_2-\ell_{\nu})}{n_2}}\\
\hat{\hat{f}}_{\beta}(x,y)=y^{n_2q} h_1(y^q)+
x^{n_2q} h_2(y^q)+\sum_{\nu\in N_{\beta}} t_{\nu}y^{\ell_{\nu}q} h_{3,\nu}(y^q) x^{\nu+q(n_2-\ell_{\nu})}\\
f_{1\beta}(x,y)=y^{n_2q} h_1(x^q y^q)+
h_2(x^q y^q)+\sum_{\nu\in N_{\beta}} t_{\nu}y^{\ell_{\nu}q} h_{3,\nu}(x^q y^q) x^{\nu}\\
f_{2\beta}(x,y)=h_1(y^q)+
x^{n_2q} h_2(y^q)+\sum_{\nu\in N_{\beta}} t_{\nu}y^{\nu} h_{3,\nu}(y^q)x^{\nu+q(n_2-\ell_{\nu})}
\end{gather*}
where $h_1(0)=(mn_1)^{n_2}$, $h_2(0)=1$ and $h_{3,\nu}(0)=(mn_1)^{\ell_{\nu}}$,
$\deg h_3(y)=(m n_1-1)\ell_\nu+n_1 b_\nu$.
For further use, $c_{ij}$ is the coefficient of $y^j$ in $h_i$, $i=1,2$ and $c_{3,\nu,j}$ for~$h_{3,\nu}$.

Let 
$$\tilde{g}(x,y)=x^{mn_1}-y^{mn_1},\tilde{\tilde{g}}(x,y)=1-y^{mn_1}$$
and define $h_4(y)$ by the property
$$
y^{\beta_3}h_4(y)=
(1-(1-y)^{mn_1})^{\beta_3}(1-y)^{n_1\beta_2-1},
$$
where $h_4(0)=(m n_1)^{\beta_3}$,
and write
$$
h_4(y)=\sum_{i=1}^{(m n_1-1) \beta_3+n_1\beta_2} b_{i}y^{i-1}.
$$
We want to compute $\res_{s=-\beta} I_{-}(f_{\beta},\beta_1,\beta_2,\beta_3)(s)$.
For $1\leq i\leq (m n_1-1) \beta_3+n_1\beta_2$, set $\nu_i$ such that
$$
\beta=\frac{n_2(m \beta_1+n_1 \beta_2+m n_1\beta_3)+q(\beta_3+i)+\nu_i}{n_2(mn_1n_2+q)};
$$
we dismiss the cases where $\nu_i<0$; note that $\nu=i q+\nu_i$.
The formula for the residue, see~\eqref{residue-B2}, is:
\begin{gather*}
\res_{s=-\beta} I_{-}(f_{\beta},\beta_1,\beta_2,\beta_3)(s)=\\
\frac{1}{n_2q}\sum_{i} \frac{1}{\nu_i!} b_{i} 
(G_{h^1_{\nu_i,-\beta,x}}(q(\beta_3+i))+G_{h^2_{\nu_i,-\beta,y}}(n_2(m\beta_1+n_1\beta_2+mn_1\beta_3-mn_1n_2\beta))
\end{gather*}
where 
$$
h^1_{\nu_i,-\beta,x}(y)=\frac{\partial^{\nu_i}f_{1\beta}^{-\beta}}{\partial x^{\nu_i}}(0,y),
\qquad h^2_{\nu_i,-\beta,y}(x)=\frac{\partial^{\nu_i}f_{2\beta}^{-\beta}}{\partial y^{\nu_i}}(x,0).
$$
We proceed as in the proof of Claim~\ref{claim1}:
\begin{gather*}
\frac{\partial^{\nu_i}f_{1\beta}^{-\beta}}{\partial x^{\nu_i}}(0,y)=
\sum_{V=(\nu_w)\in\mathcal{P}(\nu_i)}D_V
\left(\prod_{w=1}^{|V|}\frac{\partial^{\nu_w}f_{1\beta}}{\partial x^{\nu_w}}(0,y)\right)
((mn_1)^{n_2}y^{n_2q}+1)^{-\beta-\vert V\vert}\\
\frac{\partial^{\nu_i}f_{2\beta}^{-\beta}}{\partial y^{\nu_i}}(x,0)=
\sum_{V=(\nu_w)\in\mathcal{P}(\nu_i)}D_V
\left(\prod_{w=1}^{|V|}\frac{\partial^{\nu_w}f_{2\beta}}{\partial y^{\nu_w}}(x,0)\right)
((mn_1)^{n_2}+x^{n_2q})^{-\beta-\vert V\vert}
\end{gather*}
with $D_V\in \mathbb{Q}$. The derivatives without powers are computed as follows. 
For $u\in\mathbb{Z}_{\geq 0}$, let $q_u:=\left\lfloor\frac{u}{q}\right\rfloor$ and set
$$
\mathcal{C}(u):=\{h\in\{0,1,\dots,q_u\}\mid u-q h=\nu_h\in N_\beta\}
$$
\begin{gather*}
\frac{1}{u!}\frac{\partial^{u}f_{1\beta}}{\partial x^{u}}(0,y)=
\sum_{h\in \mathcal{C}(u)} c_{3,\nu_h,h}
t_{\nu_h}y^{(\ell_{\nu_h}+h)q}+\chi_\bz(K_u)(c_{K_{u}1}y^{n_2q}+c_{K_{u}2})y^{u}\\
\frac{1}{u!}\frac{\partial^{u}f_{2\beta}}{\partial y^{u}}(x,0)=\sum_{h\in \mathcal{C}(u)} c_{3,\nu_h,h}
t_{\nu_h}x^{\nu_h+q(n_2-\ell_{\nu_h})}+\chi_\bz(K_u)(c_{K_{u}1}+c_{K_{u}2}x^{n_2q})
\end{gather*}
where $K_u=\frac{u}{q}$ and $\chi_\bz$ is
the characteristic function of $\bz$.

The terms of the derivatives involved in the computation of the residues are 
parametrized by $V=(\nu_w)_{w=1}^{|V|}\in\mathcal{P}(\nu_i)$; given $V$ we decompose
its set of indices in three parts:
\begin{itemize}
\item $w\in W_1$, which determines $h_w\in\mathcal{C}(\nu_w)$, corresponding to a term with coefficient $c_{3,\nu_{h_w},h_w}$;
\item $w\in W_2$, where $\nu_w\equiv 0\bmod q$, corresponding to a term with coefficient $c_{K_{\nu_w}1}$;
\item $w\in W_3$, where $\nu_w\equiv 0\bmod q$, corresponding to a term with coefficient $c_{K_{\nu_w}2}$,
\end{itemize}
where $W=W_1\coprod W_2\coprod W_3$.
For such a 4-uple $(V,W_1,W_2,W_3)$ the integrands are
\begin{gather*}
\left(\prod_{w\in W_1}\!\!t_{\nu_{h_w}}y^{(\ell_{\nu_{h_w}}+h_w)q}\!\right)\!\!
\left(\prod_{w\in W_2}y^{\nu_w+n_2 q}\right)\!\!\left(\prod_{w\in W_3}y^{\nu_w}\right)\!
(\!(mn_1)^{n_2}y^{n_2q}\!+\!1)^{-\beta-\vert V\vert}
y^{q(\beta_3+i)}=\\
\left(\prod_{w\in W_1}t_{\nu_{h_w}}\right) y^{q(\sum_{w\in W_1} (\ell_{\nu_{h_w}}+h_w)+
\sum_{w\in W_2\cup W_3} K_{\nu_w}+n_2|W_2|+\beta_3+i)}
((mn_1)^{n_2}y^{n_2q}+1)^{-\beta-\vert V\vert}
\end{gather*}
and
\begin{gather*}
\!\left(\prod_{w\in W_1}\!\!\!t_{\nu_{h_w}}x^{\nu_{h_w}+q(n_2-\ell_{\nu_{h_w}})}\!\!\right)\!\!\!
\left(\!\prod_{w\in W_3}\!\!\!\!x^{ n_2q}\!\!\right)\!\!(\!(mn_1)^{n_2}\!\!+\!x^{n_2q})^{-\!\beta-\vert V\vert}
x^{n_2(m\beta_1\!+n_1\beta_2+mn_1\beta_3-\!mn_1n_2\beta)}\!\!=\!\\
\!\left(\prod_{w\in W_1}\!\!\!\!t_{\nu_{h_w}}\!\!\!\right)\! x^{\sum_{w\in W_1}\! (\nu_{h_w}\!+q(n_2\!-\ell_{\nu_{h_w}}\!))\!
+n_2(q|W_3|+m\beta_1\!+n_1\beta_2+mn_1\beta_3-mn_1n_2\beta)}
(\!(mn_1)^{n_2}\!\!+\!\!x^{n_2q})^{-\!\beta-\!\vert V\vert}
\end{gather*}
If $\alpha=-\beta-\vert V\vert$, we need to compute
\begin{gather*}
G_{((mn_1)^{n_2}y^{n_2q}+1)^\alpha}\left(
q\left(\sum_{w\in W_1} (\ell_{\nu_{h_w}}+h_w)+
\sum_{w\in W_2\cup W_3} K_{\nu_w}+n_2|W_2|+\beta_3+i)
\right)
\right)+\\
\!G_{((mn_1)^{n_2}\!+\!x^{n_2q})^\alpha}\!\!\left(
\sum_{w\in W_1}\!\! (\nu_{h_w}\!\!+\!q(n_2\!-\!\ell_{\nu_{h_w}}))\!
+\!n_2(q|W_3|\!\!+\!m\beta_1\!+\!n_1\beta_2\!+\!mn_1\beta_3\!-\!mn_1n_2\beta)\!\!
\right)\!
\end{gather*}
The sum of the two entries equals $-n q_2\alpha$, and by Lemma~\ref{beta}, we have
that this contribution equals
\begin{gather*}
\frac{(m n_1)^{\sum_{w\in W_1} (\ell_{\nu_{h_w}}+h_w)+
\sum_{w\in W_2\cup W_3} K_{\nu_w}+n_2|W_2|+\beta_3+i
}
}{n_2 q}\boldsymbol{B}\left(u,
v\right)
\end{gather*}
where
\begin{gather*}
u:=\frac{1}{n_2}
\left(\sum_{w\in W_1} (\ell_{\nu_{h_w}}+h_w)+
\sum_{w\in W_2\cup W_3} K_{\nu_w}+n_2|W_2|+\beta_3+i
\right)\\
v:=\frac{1}{n_2 q}\sum_{w\in W_1} (\nu_{h_w}+q(n_2-\ell_{\nu_{h_w}}))
+\frac{1}{q}(q|W_3|+m\beta_1+n_1\beta_2+mn_1\beta_3-mn_1n_2\beta)
\end{gather*}
Note that $u+v=\beta+|V|$, i.e., it is congruent with $\beta\bmod\mathbb{Z}$.

On the other side, since $q\ell_\nu\equiv\nu\bmod n_2$ the following congruences $\bmod n_2$ hold:
\begin{gather*}
\!q(\!n_2 u\!-\!(\ell_\nu\!+\!\beta_3)\!)\!\equiv\!\!\!\!
\sum_{w\in W_1} q(\ell_{\nu_{h_w}}+ h_w)+
\sum_{w\in W_2\cup W_3} \nu_w+qi-q\ell_\nu\equiv\\
\sum_{w\in W_1} (\nu_{h_w}+ qh_w)+
\sum_{w\in W_2\cup W_3} \nu_w+qi-\nu\equiv
\sum_{w=1}^{|V|} \nu_w+qi-\nu=\nu_i+q i-\nu=0;
\end{gather*}
since $\gcd(q,n_2)=1$, we deduce that
$$
u-\frac{\ell_\nu+\beta_3}{n_2}\in\bz.
$$
In particular, the corresponding contribution is, up to a factor in~$\mathbb{Q}$, the value 
$\boldsymbol{B}\left(\frac{\ell_{\nu}+\beta_3}{n_2}, \beta -\frac{\ell_{\nu}+\beta_3}{n_2}\right)$.
Hence, the total result is the product of this value and a polynomial in $\mathbf{t}$ with coefficients
in~$\mathbb{Q}$ and the coefficients of degree~$1$ do not vanish.

Let us denote $N_\beta=\{(\nu^{(1)},\dots,\nu^{(r)}\}$ where $\nu^{(1)}<\dots<\nu^{(r)}$.
For $\nu^{(1)}$, let us consider the corresponding $\beta_1^{(1)},\beta_2^{(1)},\beta_3^{(1)}$.
Then,
$$
\res_{s=-\beta} I_{-}(f_{\beta},\beta_1^{(1)},\beta_2^{(1)},\beta_3^{(1)})(s)=
\boldsymbol{B}\left(\frac{\ell_{\nu^{(1)}}+\beta_3}{n_2}, \beta -\frac{\ell_{\nu^{(1)}}+\beta_3}{n_2}\right)q^{(1)}(t_1)
$$
where $q^{(1)}$ is of degree~$1$. Hence, we can choose $t_1\in\br$ such that the above residue vanishes.
Recursively, we can choose $t_1,\dots,t_r\in\br$ such that $f_\beta$ is of
type $(n_1n_2,mn_2,mn_2+q)^-$ and 
$$
\res_{s=-\beta} I_{-}(f_{\beta},\beta_1^{(j)},\beta_2^{(j)},\beta_3^{(j)})(s)=0,\quad 1\leq j\leq r.
$$
This result does not depend on the particular choice of $(\beta_1^{(j)},\beta_2^{(j)})$.
\end{proof}


\def\cprime{$'$}
\providecommand{\bysame}{\leavevmode\hbox to3em{\hrulefill}\thinspace}
\providecommand{\MR}{\relax\ifhmode\unskip\space\fi MR }
\providecommand{\MRhref}[2]{%
  \href{http://www.ams.org/mathscinet-getitem?mr=#1}{#2}
}
\providecommand{\href}[2]{#2}

\end{document}

\end{document}